\documentclass[11pt]{article}

\pdfoutput=1

\usepackage{amsmath,amsfonts,amssymb,latexsym,amsthm,dsfont}
\usepackage[a4paper,margin=3cm]{geometry}
\usepackage{graphicx}
\usepackage[backref]{hyperref}
\usepackage[utf8]{inputenc}
\usepackage[T1]{fontenc}
\usepackage{lmodern,microtype}
\usepackage{tikz,enumitem}
\usetikzlibrary{decorations.markings,calc,intersections}
\usepackage{ifthen}

\newtheorem{thm}{Theorem}[section]%
\newtheorem{cor}[thm]{Corollary}%
\newtheorem{prop}[thm]{Proposition}%
\newtheorem{lem}[thm]{Lemma}%

\newtheorem{defi}[thm]{Definition}%
\newtheorem{rem}[thm]{Remark}%

\newcommand{\DeclareAlphabet}[2]{%
  \foreach \x in {A,B,...,Z}{%
    \expandafter\xdef
    \csname #1\x\endcsname{%
      \noexpand#2{\x}}%
  }
}

\DeclareAlphabet{d}{\mathbb}  
\DeclareAlphabet{c}{\mathcal}
\DeclareAlphabet{r}{\mathrm}
\DeclareAlphabet{b}{\mathbf}

\newcommand{\p}[4]{{#3}\!\left#1{#4}\right#2}
\newcommand{\ABS}[1]{\left| #1 \right|} 
\newcommand{\BRA}[1]{{\left\{ #1 \right\}}} 
\newcommand{\NRM}[1]{\left\| #1\right\|} 
\newcommand{\PAR}[1]{{\left( #1 \right)}} 
\newcommand{\SBRA}[1]{{\left[#1\right]}} 
\let\TV\VT

\renewcommand{\leq}{\leqslant}
\renewcommand{\geq}{\geqslant}
\newcommand{\eps}{\varepsilon}
\renewcommand{\epsilon}{\varepsilon}

\newcommand{\tpullbk}[2]{%
\ifthenelse{\equal{#1}{}}{%
  \Phi^{\star}_{#2}}{%
  \Phi^{#1,\star}_{#2}}%
}
\newcommand{\tpushfw}[2]{%
\ifthenelse{\equal{#2}{}}{%
  \Phi^{#1}_{\star}}{%
  \Phi^{#1}_{#2,\star}}%
}
\newcommand{\pullbk}[1]{%
  \bphi^{\star}_{#1}}{%
}

\newcommand{\prb}[2][]{\dP_{#1}\left[#2\right]}
\newcommand{\esp}[2][]{\dE_{#1}\left[#2\right]}

\newcommand{\Cdiscret}{\tilde{\mathcal{C}}}
\newcommand{\Ccontinu}{\mathcal{C}}

\newcommand{\ind}{\mathds{1}}
\newcommand{\leb}[1]{\lambda_{\dR^{#1}}}

\newcommand{\vtr}[1]{\mathbf{#1}}

\DeclareMathOperator{\rk}{rank}

\DeclareMathOperator{\Div}{div}
\DeclareMathOperator{\supp}{supp}
\newcommand{\famF}{\mathcal{F}}
\newcommand{\famG}{\mathcal{G}}
\newcommand{\famH}{\mathcal{H}}
\newcommand{\PhiIU}{\bphi^{\vtr{i}}_{\vtr{u}}}
\newcommand{\PhiIV}{\bphi^{\vtr{i}}_{\vtr{v}}}
\newcommand{\FIU}{\mathcal{F}_{\vtr{i},\vtr{u}}}
\newcommand{\FIUTilde}{\mathcal{F}_{\tilde{\vtr{i}},\tilde{\vtr{u}}}}
\newcommand{\GIU}{\mathcal{G}_{\vtr{i},\vtr{u}}}
\newcommand{\GIUTilde}{\mathcal{G}_{\tilde{\vtr{i}},\tilde{\vtr{u}}}}
\newcommand{\tribuDuPoisson}{\cF_{Poi}}

\def\bi{\vtr{i}}
\def\bj{\vtr{j}}
\def\bv{\vtr{v}}
\def\bu{\vtr{u}}
\def\bt{\vtr{t}}
\def\bs{\vtr{s}}

\def\bphi{\mathbf{\Phi}}
\def\Rp{\dR_+}
\def\Rpn{\dR_+^n}

\def\tZ{\tilde{Z}}
\def\tX{\tilde{X}}
\def\tY{\tilde{Y}}
\def\tP{\tilde{P}}
\def\tK{\tilde{K}}
\def\TT{\mathbb{T}}

\title{Qualitative properties of certain piecewise deterministic Markov processes}
\author{Michel Benaïm, Stéphane Le Borgne, Florent Malrieu, Pierre-André Zitt}

\begin{document}

\maketitle

\begin{abstract}
We study a class
of Piecewise Deterministic Markov Processes with state space
$\dR^d\times E$ where $E$ is a finite set. The continuous
component evolves according to a smooth vector field that is switched
at the jump times of the discrete coordinate. The jump rates may depend
on the whole position of the process. Working under the general assumption
that the process stays in a compact set, we detail a possible construction
of the process and characterize its support, in terms of the solutions
set of a differential inclusion. We establish results on the long time behaviour
of the process, in relation to a certain set of accessible points, which is shown
to be strongly linked to the support of invariant measures. Under Hörmander-type
bracket conditions, we prove that there exists a unique invariant measure
and that the processes converges to equilibrium in total variation.
Finally we give examples where the bracket condition does not hold, and where
there may be one or many invariant measures, depending on the jump rates
between the flows.

\medskip

\medskip

\end{abstract}

\section{Introduction}

Piecewise deterministic Markov processes (PDMPs in short) are intensively
used in many applied areas (molecular biology \cite{RMC}, storage modelling
 \cite{boxma}, Internet traffic \cite{MR1895332,MR2588247,GR2}, neuronal
  activity \cite{riedler,PTW},...).
 Roughly speaking, a Markov process is a PDMP if its randomness is only
 given by the jump mechanism: in particular, it admits no diffusive
 dynamics. This huge class of processes has been introduced by
 Davis \cite{davis}. See \cite{MR1283589,jacobsen} for a general presentation.
 
 In the present paper, we deal with an interesting subclass of the
 PDMPs that plays a role in molecular biology \cite{RMC, riedler} 
(see also \cite{yin} for other motivations). We consider a PDMP evolving  on
 $\dR^d\times E$, where $d\geq 1$ and $E$ is a finite set, as follows: the
 first coordinate moves continuously on $\dR^d$ according to
 a smooth vector field that depends on the second coordinate, whereas the
 second coordinate jumps with a rate depending on the first one. 
 Of course, most of the results in the present paper should  extend to smooth 
 manifolds. This class of Markov processes is reminiscent of the
 so-called iterated random functions in the discrete time setting
 (see \cite{diaconis} for a good review of this topic).

 We are interested in the long time qualitative behaviour of these processes.
 A recent paper by Bakhtin and Hurth \cite{bakhtin&hurt} considers the
 particular situation where the jump rates are constant and prove the
 beautiful result that, under a Hörmander type condition, if there exists
 an invariant measure for the process, then it is unique and absolutely
 continuous with respect to the ``Lebesgue'' measure on $\dR^d \times E$.
 Here we consider a more general situation and focus also on the convergence 
 to equilibrium. We also provide a basic proof of the main result in \cite{bakhtin&hurt}. 

 Let us define our process more precisely.
 Let $E$ be a finite set, and for any $i\in E$, $F^i: \dR^d \mapsto \dR^d$
 be a smooth vector field. We assume throughout that each $F^i$ is bounded and 
 we denote by $C_{sp}$ an upper bound for the ``speed'' of the deterministic 
 dynamics: 
 \[
\sup_{x\in\dR^d,i\in E} \|F^i(x)\| \leq C_{sp} < \infty.
\]
We let $\Phi^i = \{\Phi^i_t\}$
 denote the flow induced by $F^i$. Recall that
  \[
  t \mapsto \Phi^i_t(x) = \Phi^i(t,x)
  \]
  is the solution to the Cauchy problem $\dot{x} = F^i(x)$ with initial condition
$x(0) = x$. Moreover, we assume  that there exists a compact
set~$M \subset \dR^d$ that is  positively invariant under each $\Phi^i$, meaning that:
\begin{equation}
  \label{eq=invarianceAssumption}
  \forall i\in E,\ \forall t \geq 0, \quad  \Phi^i_t(M) \subset M.
\end{equation}
 We consider here a continuous time Markov process $(Z_t = (X_t,Y_t))$
 living on $M \times E$ whose infinitesimal generator acts on functions
\begin{align*}
 g: M \times E&\rightarrow \mathbb{R},\\
 (x,i) &\mapsto g(x,i) = g^i(x),
\end{align*}
smooth\footnote{meaning that $g^i$ is the restriction to $M$ of a smooth
function on $\dR^d$.} in $x$, according to the formula
\begin{equation}\label{eq:gen-inf}
Lg(x,i)= \langle F^i(x), \nabla g^i(x)\rangle +
\sum_{j\in E}\lambda(x,i,j)(g^j(x)-g^i(x))
\end{equation}
where
\begin{enumerate}[label={\bfseries(\roman*)}]
  \item $x \mapsto \lambda(x,i,j)$ is continuous;
  \item $\lambda(x,i,j) \geq 0$ for $i \neq j$ and $\lambda(x,i,i) = 0;$
  \item for each $x \in M$, the matrix $(\lambda(x,i,j))_{ij}$ is  irreducible.
\end{enumerate}

The process is explicitly constructed in Section~\ref{sec:construct} and
some of its basic properties (dynamics, invariant and empirical occupation 
probabilities) are established. In Section~\ref{sec:law} we
describe  (Theorem \ref{supportlaw}) the support of the law of the process
in term of the solutions set of a differential inclusion induced by the collection
$\{F^i \: : i \in E \}.$ Section~\ref{sec:acces} introduces the \emph{accessible
set} which is a natural candidate to support invariant probabilities. We
show (Proposition~\ref{limsetprop}) that this set is compact, connected,
strongly positively invariant and invariant under the differential inclusion
induced by $\{F^i \: : i \in E \}.$ Finally, we prove that, if the process has a 
unique invariant probability measure, its support is characterized in terms 
of the accessible set. 

Section~\ref{sec:reg-and-erg} contains the main results of the present paper. 
We begin by a slight improvement of the regularity results of \cite{bakhtin&hurt}: under Hörmander-like 
bracket conditions, the law of the process after a large enough number of jumps 
or at a sufficiently large time has an absolutely continuous component 
with respect to the Lebesgue measure on $\dR^d\times E$. Moreover, this component may be 
chosen uniformly with respect to the initial distribution. The proofs of these results 
are postponed to Sections~\ref{absolute} and \ref{alamain}.
We use these estimates in Section~\ref{sec:erg} to 
establish the exponential ergodicity (in the sense of the total variation distance) of 
the process under study. In Section~\ref{sec:elementary}, we 
show that our assumptions are sharp thanks to several examples. In particular, 
we stress that, when the Hörmander condition is violated, the uniqueness of the
invariant measure may depend on the jump mechanism between flows,
and not only on the flows themselves.

\begin{rem}[Quantitative results]
In the present paper, we essentially deal with qualitative properties of the asymptotic behavior  
for a large class of PDMPs. Under more stringent assumptions, \cite{BLMZquant} gives an
explicit rate of convergence in Wasserstein distance, via a coupling argument.
\end{rem}

\begin{rem}[Compact state space]
The main results in the present paper are still valid even if the state space is no longer compact 
provided that the excursions out of some compact sets are suitably controlled (say
with a Lyapunov function). Nevertheless, as shown in~\cite{BLMZexample}, the stability of 
Markov processes driven by an infinitesimal generator as \eqref{eq:gen-inf} may depend on the 
jump rates. As a consequence, it is difficult to establish, in our general framework, sufficient 
conditions for the stability of the process under study without the invariance 
assumption~\eqref{eq=invarianceAssumption}; results in this direction may however be found
in the recent~\cite{CH13}. 
\end{rem}


\section{Construction and basic properties}
\label{sec:construct}

In this section we explain how to construct explicitly the process $(Z_t)_{t\geq 0}$ 
driven by \eqref{eq:gen-inf}.
Standard references for the construction and properties of more general PDMPs 
are the monographs \cite{MR1283589} and \cite{jacobsen}. In our case the compactness
allows a nice construction via a discrete process whose jump times follow
an homogeneous Poisson process, similar to the  classical ``thinning'' 
method for simulating non-homogeneous Poisson processes (see~\cite{LS79,Ros97}). 

\subsection{Construction}

Since $M$ is compact and the maps $\lambda(\cdot,i,j)$ are continuous,
there exists $\lambda \in \dR_+$ such that
\[
\max_{x \in M, i\in E } \sum_{j\in E, j \neq i} \lambda(x,i,j)<\lambda.
\]
Let us fix such a $\lambda$, and let
\[
Q(x,i,j) = \frac{\lambda(x,i,j)}{\lambda}
\text{, for } i \neq j \qquad \text{and}\qquad
Q(x,i,i) = 1 - \sum_{j \neq i} Q(x,i,j) >0. 
\]
Note that $Q(x)={(Q(x,i,j))}_{i,j\in E}$ is an irreducible aperiodic Markov transition matrix
and that \eqref{eq:gen-inf} can be rewritten as
\[
Lg = Ag + \lambda (Q g - g)
\]
where
\begin{equation}  \label{defA}
Ag(x,i) = \langle F^i(x), \nabla g^i(x)\rangle 
\quad \text{and}\quad
Q g(x,i) =  \sum_{j\in E}Q(x,i,j)g^j(x). 
\end{equation}
Let us first construct a discrete time Markov chain ${(\tZ_n)}_{n\geq 0}$. 
Let ${(N_t)}_{t\geq 0}$ be  a homogeneous Poisson 
with intensity $\lambda$;  denote by ${(T_n)}_{n\geq 0}$ its jump times 
and ${(U_n)}_{n\geq 0}$ its interarrival times. 
Let $\tZ_0 \in M \times E$ be a random 
variable independent of ${(N_t)}_{t\geq 0}$.
Define  ${(\tZ_n)}_n = {(\tX_n,\tY_n)}_n$ on
$M \times E$ recursively by:
\begin{align*}
  \tX_{n+1} &= \Phi^{\tY_n} \PAR{U_{n+1},\tX_n},\\
  \prb{  \tY_{n+1} = j \middle|  \tX_{n+1}, \tY_n = i} &= Q \PAR{\tX_{n+1},i,j}.
\end{align*}
Now define ${(Z_t)}_{t\geq 0}$ via interpolation by setting
\begin{equation}
\label{defz}
\forall  t \in [T_n,T_{n+1}),\quad 
  Z_t= \PAR{\Phi^{\tY_n}\PAR{t-T_n,\tX_n}, \tY_n}.
\end{equation} 
The memoryless property of exponential random variables makes
${(Z_t)}_{t\geq 0}$ a continuous time càdlàg Markov process. 
We let $P = {(P_t)}_{t \geq 0}$ denote the semigroup induced by
${(Z_t)}_{t \geq 0}$. 
Denoting by $\cC_0$ (resp. $\cC_1$)  the set of real valued functions
$f : M \times E \rightarrow \dR$ that are continuous (resp continuously
differentiable) in the first variable, we have the following result. 
\begin{prop}
The infinitesimal generator of the semigroup $P = {(P_t)}_{t\geq 0}$ is the operator $L$ 
given by~\eqref{eq:gen-inf}. Moreover, $P_t$ is \emph{Feller},
meaning that it maps $\cC_0$ into itself and, for $f\in\cC_0$,  $\lim_{t \rightarrow 0} \NRM{P_t f - f} = 0$.

The transition operator $\tP$ of the Markov chain $\tZ$ also maps
$\cC_0$ to itself, and if $K_t$ and  $\tK$ are defined by
\begin{equation}
  K_t g(x,i) = g\PAR{\Phi_t^i(x),i}
\quad\text{and}\quad
  \label{deftK}
\tK f = \int_0^{\infty} \lambda e^{-\lambda t}K_t f dt,
\end{equation}
then $\tP$ can be written as:
\begin{align}
\label{defP}
\tP g(x,i) 
&= \esp{g(\tZ_1)\middle| Z_0 = (x,i)}   \notag
= \int_0^{\infty} K_tQ g(x,i) \lambda e^{-\lambda t} dt \\
&= \tK Q g(x,i).
\end{align}
\end{prop}

\begin{proof}
For each $t \geq 0,$ $P_t$ acts on bounded
measurable maps $g : M \times E \rightarrow \dR$ according to the formula
\begin{equation*}
  P_t g(x,i) = \esp{g(Z_t)|Z_0 = (x,i)}.
\end{equation*}
For $t\geq 0$ let $J_t = K_tQ$. It follows from  \eqref{defz} that
\begin{equation}
\label{eq:def2Pt}
P_t g  = \sum_{n \geq 0} \esp{ \ind_{\{N_t = n\}} J_{U_1} \circ\cdots\circ J_{U_n} \circ K_{t -T_n} g }
\end{equation}

By Lebesgue continuity theorem and \eqref{eq:def2Pt}, $P_tg  \in \cC_0$
whenever $g  \in \cC_0$. Moreover, setting apart the first two terms
in~\eqref{eq:def2Pt} leads to
\begin{equation}
  \label{eq=deuxTermes}
  P_t g = e^{-\lambda t} K_t g  +
  \lambda  e^{-\lambda t} \int_0^t K_u Q K_{t-u} g \,  du + R(g,t)
\end{equation}
where $\ABS{R(g,t)} \leq \NRM{g} \prb{N_t > 1}
= \NRM{g} (1 - e^{-\lambda t}(1 + \lambda t)).$
Therefore $\lim_{t \rightarrow 0} \NRM{P_t g - g} = 0$.

The infinitesimal generator of $(K_t)$ is the operator $A$ defined by~\eqref{defA}. Thus 
\(
 \frac{1}{t}(K_t g - g ) \to Ag,
\)
therefore, by \eqref{eq=deuxTermes},
\[
  \frac{P_t g - g}{t} \xrightarrow[t\to 0]{}  A g - \lambda g + \lambda Qg,
\]
and the result on $(P_t)$ follows.
The expression~\eqref{defP} of $\tP$ is a consequence
of the definition of the chain. From~\eqref{defP} one
can deduce that $\tP$ is also  Feller.
\end{proof}

\begin{rem}
  [Discrete chains and PDMPs] The chain~$\tZ$ records all jumps of the 
  discrete part~$Y$ of the PDMP~$Z$, but, since $Q(x,i,i)>0$, 
  $\tZ$ also contains ``phantom jumps'' that cannot be seen 
  directly on the trajectories of~$Z$. 

  Other slightly different discrete chains may crop up in the study of PDMPs. 
  The most natural one is the process observed 
  at (true) jump times. In another direction, the chain~$(\Theta_n)_{n\in\dN}$ 
  introduced in~\cite{MR2385873} corresponds (in our setting) to the
  addition of phantom jumps at rate~$1$. For this chain~$(\Theta_n)$, 
  the authors prove (in a more general setting) equivalence between 
  stability properties of the discrete and continuous time processes. 

  Similar equivalence properties will be shown below for our chain~$\tZ$
  (Proposition~\ref{homeo}, Lemma~\ref{martinempir}). 
  Its advantage lies in the simplicity of its definition;
  in particular it leads to a simulation method that does not
  require the integration of jump rates along trajectories. 
\end{rem}

\paragraph{Notation.} Throughout the paper we may write $\prb[x,i]{\cdot}$
for $\prb{\cdot\middle |Z_0 = (x,i)}$ and  $\esp[x,i]{\cdot}$ for $\esp{\cdot|Z_0 = (x,i)}$.

\subsection{First properties of the invariant probability measures} 

Let $\cM(M \times E)$ (respectively  $\mathcal{M}^+(M \times E)$
and $\mathcal{P}(M \times E)$) denote the set of signed (respectively
positive, and probability) measures on $M \times E$. For
$\mu \in \mathcal{M}(M \times E)$ and $f \in L^1(\mu)$ we write
$\mu f$  for $\int f d\mu$. Given a  bounded operator $K : \cC_0 \to \cC_0$
and $\mu \in \mathcal{M} (M \times E)$ we let
$\mu K \in \mathcal{M}(M \times E)$ denote  the measure defined by duality :
\[
\forall g \in \cC_0, \quad (\mu K)g=  \mu(K g).
\]
The mappings  $\mu \mapsto \mu P_t$ and $\mu\mapsto  \mu \tP$ preserve the sets
$\mathcal{M}^+(M \times E)$ and  $\mathcal{P}(M \times E)$.

\begin{defi}[Notation, stability]
  We denote by $\mathcal{P}_{inv}$ (resp. $\tilde{\mathcal{P}}_{inv}$)
 the set of invariant probabilities for $(P_t)$ (resp. $\tP$), 
 \begin{align*}
  \mu \in \mathcal{P}_{inv} &\iff \forall t \geq 0,\quad\mu P_t = \mu ; \\
  \mu \in \tilde{\mathcal{P}}_{inv} &\iff \mu \tP = \mu. 
\end{align*}
We say that the process $Z$ (or $\tilde Z$) is \emph{stable} if it has a 
unique invariant probability measure.

For $n \in \mathbb{N}^*$ and $t > 0$ we let $\tilde{\Pi}_n$ and $\Pi_t$ the
(random) occupation measures defined by
\[
\tilde{\Pi}_n = \frac{1}{n} \sum_{k = 1}^n \delta_{\tZ_k}
\quad\text{and}\quad
 \Pi_t = \frac{1}{t} \int_0^t \delta_{Z_t}.
 \]
\end{defi}

By standard results for Feller chains on a compact space
(see e.g.~\cite{duf00}), the set $\tilde{\mathcal{P}}_{inv}$ is non empty, 
compact (for the weak-$\star$ topology) and convex. Furthermore, 
with probability one every limit point of ${(\Pi_n)}_{n\geq 1}$ lies in 
$\tilde{\mathcal{P}}_{inv}$.

The following result gives an explicit correspondence between invariant 
measures for the discrete and continuous processes. 

\begin{prop}[Correspondence for invariant measures]
\label{homeo}
The mapping $\mu \mapsto \mu \tK$ maps~$\tilde{\mathcal{P}}_{inv}$
homeomorphically onto~$\mathcal{P}_{inv}$ and extremal points of%
~$\tilde{\mathcal{P}}_{inv}$ (i.e. ergodic probabilities for~$\tP$)  onto
extremal points of~$\mathcal{P}_{inv}$ (ergodic probabilities for~$P_t$). 

The inverse homeomorphism is the map   $ \mu  \mapsto \mu Q$
restricted to ${\mathcal{P}}_{inv}$. 

Consequently, the continuous time process $(Z_t)$ is stable if and only
if the Markov chain $(\tZ_n)$ is stable. 
\end{prop}

\begin{proof}
 For all $f \in \cC_1,$ integrating by parts
 $\int_0^{\infty} \frac{d K_t f}{dt} e^{-\lambda t} dt$ and using the identities
 $\frac{d K_t f}{dt} = A K_t f = K_t A f$ leads to
\begin{equation}
\label{pseudoinv}
\tK (\lambda I - A) f = \lambda f = (\lambda I - A) \tK f.
\end{equation}
Let $\mu \in  \mathcal{P}(M \times E).$ Then, using \eqref{pseudoinv}
and the form of $L$ gives
\begin{align}
\label{mutkl}
\mu \tK L f
&= \mu \tK (A-\lambda I)  f + \lambda \mu \tK  Q f = \lambda \PAR{- \mu f +  \mu \tP f},
\\
\label{multk}
\mu L (\tK f)
&= \mu (A-\lambda I ) \tK f + \lambda \mu Q \tK f = \lambda \PAR{-\mu f + \mu Q \tK f}.
\end{align}
If $\mu \in \tilde{\mathcal{P}}_{inv}$, \eqref{mutkl} implies  $(\mu \tK) L f = 0$
for all $f \in \cC_1$ and  since $\cC_1$ is dense in $\cC_0$ this proves that
$\mu \tK \in  {\mathcal{P}}_{inv}$.  Similarly, if
$\mu \in {\mathcal{P}}_{inv}$, (\ref{multk}) implies $\mu = \mu Q \tK$.
Hence $(\mu Q) = (\mu Q) \tK Q = (\mu Q) \tP$ proving that
$\mu Q \in \tilde{\mathcal{P}}_{inv}.$ Furthermore the identity
$\mu = \mu Q \tK$ for all  $\mu \in {\mathcal{P}}_{inv}$ shows that the maps
$\mu  \mapsto \mu \tK$ and $\mu  \mapsto \mu Q$ are inverse homeomorphisms.
\end{proof}

\begin{lem}[Comparison of empirical measures]
\label{martinempir}
Let $f : M \times E \mapsto \dR$ be a bounded measurable function.
Then
\[
\lim_{t \rightarrow \infty} \Pi_t f  - \tilde{\Pi}_{N_t} \tilde{K} f = 0
\quad \text{and} \quad
\lim_{n \rightarrow \infty} \Pi_{T_n} f  - \tilde{\Pi}_{n} \tilde{K} f = 0
\]
with probability one.
\end{lem}

\begin{proof}
  Decomposing the continuous time interval $[0,t]$ along the jumps yields:
\[\Pi_t f =
\frac{N_t}{t} \PAR{  \frac{1}{N_t} \sum_{i = 0}^{N_{t}-1}
\int_{T_i}^{T_{i+1}} f(Z_s)ds + r_t }
\]
where $\|r_t\| \leq \|f\| \frac{U_{N_t+1}}{N_t}$.

Since  $\lim_{t\rightarrow\infty} \frac{N_t}{t} = \lambda$ almost surely  and
$\prb{U_n/n \geq \eps } = e^{-\lambda n \eps}$,
$r_t \xrightarrow[t\to\infty]{a.s.}0$,  one has
  \[
\Pi_t f - \frac{1}{N_t} \sum_{i = 0}^{N_{t}-1} \int_{T_i}^{T_{i+1}} f(Z_s)ds
 \xrightarrow[t\to\infty]{a.s.} 0.
  \]
   Now, note that
 \[
   \int_{T_i}^{T_{i+1}} f(Z_s) ds =  \int_0^{U_{i+1}} f\PAR{\phi_s^{\tY^i}(\tX_i),\tY_i} ds,
 \]
therefore
 \[
   M_n = \sum_{i = 0}^{n-1} \PAR{\int_{T_i}^{T_{i+1}} f(Z_s) ds - \tK f (\tX_i, \tY_i)}
 \]
 is a martingale with increments bounded in $L^2$:
 $\esp{(M_{n+1}-M_n)^2} \leq 2 \|f\|^2/\lambda^2 .$
Therefore, by the strong law of large numbers for martingales,
\(\lim_{n\rightarrow\infty} \frac{M_n}{n} = 0\) almost surely, 
and the result follows. 
\end{proof}

As in the discrete time framework, the set $\mathcal{P}_{inv}$  is non empty compact 
(for the weak-$\star$ topology) and convex. Furthermore, with probability one,
every limit point of ${(\Pi_t)}_{t\geq 0}$ lies in $\mathcal{P}_{inv}$.

Finally, one can check that an invariant measure for the embedded chain and its 
associated invariant measure for the time continuous process have the same 
support. Given $\mu \in \cP(M \times E)$ let us denote by $\supp(\mu)$ its support.
\begin{lem}\label{samesupport}
Let $\mu \in \tilde{\mathcal{P}}_{inv}.$ Then $\mu$ and $\mu \tK$ have
the same support. 
\end{lem}
\begin{proof} Let $(x,i) \in \supp(\mu)$ and let $\cU$ be a neighborhood of $x$.
Then for $t_0 > 0$ small enough and $0 \leq t \leq t_0$,  $\Phi^i_{-t}(\cU)$ is
also a neighborhood of $x$.  Thus
\[
(\mu \tK)(\cU \times \{i\}) = \int_0^\infty\! \lambda e^{-\lambda t} \mu (\Phi^i_{-t}(\cU)\times \{i\})\, dt
\geq \lambda \int_0^{t_0} \!e^{-\lambda t} \mu (\Phi^i_{-t}(\cU)\times \{i\})\, dt > 0.
\]
This proves that $\supp( \mu) \subset  \supp(\mu \tK)$. Conversely, let 
$\nu = \mu \tK$ and $(x,i) \in \supp(\nu)$ and let $\cU$ be a neighborhood of $x$.
Then
\[
\mu(\cU\times \{i\}) = (\nu Q)(\cU \times\{i\}) = \sum_{j\in E} \int_\cU Q_{ji}(x) \nu(dx \times \{j\})
\geq \int_\cU Q_{ii}(x) \nu(dx \times \{i\}) > 0.
\]
As a consequence, $\supp( \mu) \supset  \supp(\mu \tK)$.
\end{proof}

\subsection{Law of pure types}

Let $\lambda_M$ and $\lambda_{M\times E}$ denote the Lebesgue measures on $M$ and 
$M \times E$.

\begin{prop}[Law of pure types]
Let $\mu \in \tilde{{\mathcal P}}_{inv} $  (respectively ${{\mathcal P}}_{inv}$)
and  let $\mu = \mu_{ac} + \mu_s$ be the Lebesgue decomposition of $\mu$
with $\mu_{ac}$ the absolutely continuous (with respect to $\lambda_{M \times E}$)
measure and $\mu_s$ the singular (with respect to $\lambda_{M \times E}$) measure.
Then both $\mu_{ac}$ and $\mu_s$ are in $\tilde{{\mathcal P}}_{inv}$
(respectively ${{\mathcal P}}_{inv}$), up to a multiplicative constant.  In particular, if $\mu$ is ergodic,
then $\mu$ is either absolutely continuous or singular.
\end{prop}

\begin{proof}
The key point is that $\tK$ and $Q$, hence $\tP = \tK Q,$ map
absolutely continuous measures into absolutely continuous measures.
For $\mu \in \tilde{{\mathcal P}}_{inv}$ the result now follows from the
following simple Lemma \ref{typepur} applied to $\tP.$
\end{proof}

\begin{lem}
\label{typepur}
Let $(\Omega, \mathcal{A},P)$ be a probability space. Let $\mathcal{M}$ (respectively
$\mathcal{M}^+, \mathcal{P}, \mathcal{M}_{ac})$ denote the set of
signed (positive, probability, absolutely continuous with respect to $P$) measures on $\Omega$.
Let $K : \mathcal{M} \to \mathcal{M}, \mu \mapsto \mu K$ be a linear
map that maps each of the preceding sets into itself. Then if
$\mu \in \mathcal{P}$ is a fixed point for $K$ with Lebesgue decomposition
$\mu = \mu_{ac} + \mu_s$, both $\mu_{ac}$ and $\mu_s$ are fixed points for $K$.
\end{lem}

\begin{proof}
Write $\mu K = \mu_{ac} K + \mu_s K = \mu_{ac} K + \nu_{ac} + \nu_{s}$
with $\mu_s K = \nu_{ac} + \nu_s$ the Lebesgue decomposition of
$\mu_s K$. Then, by uniqueness of the decomposition,
$\mu_{ac} = \mu_{ac} K + \nu_{ac}$. Thus, $\mu_{ac} \geq \mu_{ac} K$.
Now either $\mu_{ac} = 0$ and there is nothing to prove or, we can
normalize by $\mu_{ac}(\Omega)$ and we get that  $\mu_{ac} = \mu_{ac} K$.
\end{proof}


\section{Supports and accessibility}

\subsection{Support of the law of paths}
\label{sec:law}

In this section, we describe the shape of the support of the distribution of ${(X_t)}_{t\geq 0}$ 
and we show that it can be linked to the set of solutions of a differential inclusion (which is a 
generalisation of ordinary differential equations). 

Let us start with a definition that will prove useful to encode the paths
of the process~$Z_t$. 
\begin{defi}[Trajectories and adapted sequences]
  \label{def:adapted}
For all $n  \in \mathbb{N}^*$ let 
\[
\TT_n = \BRA{(\bi,\bu)=((i_0, \ldots, i_n),(u_1, \ldots, u_n)) \in E^{n+1} \times \Rpn}
\]
and 
\[
\TT^{i,j}_n = \BRA{(\bi,\bu)\in\TT_n\ :\ i_0=i,\, i_n=j}.
\]
Given $(\bi,\bu) \in \TT_n$ and $x \in M$, define $(x_k)_{0\leq k\leq n}$ by induction
by setting $x_0 = x$ and $x_{k+1} = \Phi^{i_{k-1}}_{u_k}(x_k)$: these are the points 
obtained by following $F^{i_0}$ for a time $u_0$, then $F^{i_1}$ for a time $u_1$, etc.

We also define a corresponding continuous trajectory.  Let $\bt=(t_0,\ldots,t_n)$ be defined  by 
$t_0=0$ and $t_k=t_{k-1}+u_k$ for $k=1,2,\ldots,n$ and let ${(\eta_{x,\bi,\bu}(t))}_{t\geq 0}$ 
be the function ${(\eta(t))}_{t\geq 0}$ given by 
\begin{equation}\label{def:etaiu}
\eta_{x,\bi,\bu}(t) = \eta(t)=
\begin{cases}
x&\text{if }t=0,\\
\Phi^{i_{k-1}}_{t-t_{k-1}}(x_{k-1}) &\text{if }t_{k-1}< t\leq t_k \text{ for }k=1,\ldots,n,\\ 
\Phi^{i_n}_{t-t_n}(x_n) &\text{if }t>t_n.
\end{cases}
\end{equation}
Finally, let 
\(
p(x,\bi,\bu) = Q(x_1,i_0,i_1) Q(x_2,i_1,i_2)
\cdots Q(x_n,i_{n-1},i_n),
\)
 and
\[
\dT_{n,\mathrm{ad}(x)}=\BRA{(\bi,\bu)\in\TT_n\ :\ p(x,\bi,\bu) > 0}.
\]
An element of $\dT_{n,\mathrm{ad}(x)}$ is said to be  \emph{adapted to}
$x \in M$.
\end{defi}

\begin{lem}\label{le:tube}
Let $(\bi,\bu)\in\TT_n$. Then, for any $(x,i)\in M\times E$, any $T\geq 0$,  
and any $\delta >0$, 
\[
  \prb[x,i]{\sup_{0\leq t\leq T}\NRM{X_t-\eta_{x,\bi,\bu}(t)}\leq \delta}>0.
\]
\end{lem}
\begin{proof}
Suppose first that $(\bi,\bu)\in\TT_{n}$ is adapted to $x$ and that $\bi$ starts at $i$. 
By continuity, there exist $\delta_1$ and $\delta_2$ 
such that 
\[
\max_{i = 1,\ldots, n} |s_i-u_i| \leq \delta_1
\quad\Longrightarrow \quad
\begin{cases}
  \sup_{0\leq t \leq T} \NRM{\eta_{x,\bi,\bs}(t) - \eta_{x,\bi,\bu}(t)} \leq \delta, \\
p(x,\bi,\bs)\geq \delta_2.
\end{cases}
\]
Let $(U_1,\ldots,U_n,U_{n+1})$ be $n+1$ independent random variables with an 
exponential law of parameter $\lambda$ and $\bU=(U_1,\ldots,U_n)$. Then,
\begin{align*}
  &\prb[x,i] {\sup_{0\leq t\leq T}\NRM{X_t-\eta_{x,\bi,\bu}(t)}\leq \delta}\\
  &\qquad
  \geq \prb{\max_{l = 1,\ldots, n} |U_l-u_l | \leq \delta_1,
            \ U_{n+1}\geq T-t_n+\delta_1, 
            (\tilde Y_0,\ldots,\tilde Y_n)=\bi
	  }\\
  &\qquad
  \geq  \delta_2\prb{ \max_{l = 1,\ldots, n} |U_l-u_l | \leq \delta_1,\ 
                      U_{n+1}\geq T-t_n+\delta_1
		    }\\
  &\qquad
  \geq \delta_2 \SBRA{\prod_{l=1}^n \PAR{e^{-\lambda(u_l-\delta_1)}-e^{-\lambda(u_l+\delta_1)}}} 
                e^{-\lambda (T-t_n+\delta_1)} 
  > 0.
\end{align*}
In the general case,  $(\bi,\bu)\in\TT_n$ is not necessarily adapted and may start at
an arbitrary $i_0$. However, for any $T>0$ and $\delta>0$, there exists 
$(\bj,\bv)\in\TT_{n',\mathrm{ad}(x)}$ (for some $n'\geq n$) such that $j_0=i$ and  
\[
\sup_{0\leq t\leq T}\NRM{\eta_{x,\bj,\bv}(t)-\eta_{x,\bi,\bu}(t)}\leq \delta,
\]
since $Q(x)$ is, by construction, irreducible and aperiodic (this allows to 
add permitted transitions from $i$ to $i_0$ and where ${\bi}$ has not permitted transitions, 
with times between the jumps as small as needed). 
\end{proof}

After these useful observations, we can describe the support of the law of  ${(X_t)}_{t \geq 0}$  
in terms of a certain differential inclusion induced by the vector fields $\{F^i\, : \, i\in E\}$.

For each $x \in \dR^d,$ let  $\mathrm{co}(F)(x) \subset \dR^d$ be the compact
convex  set defined as
\[
  \mathrm{co}(F)(x) = \BRA{ \sum_{i \in E } \alpha_i F^i(x) :
  \: \alpha_i \geq 0, \sum_{i \in E} \alpha_i = 1  }.
\]
Let $\cC(\dR_+,\dR^d)$ denote the set of continuous paths
$\eta : \dR_+ \mapsto \dR^d$ equipped with the topology of uniform
convergence on compact intervals. A solution to the differential inclusion
\begin{equation}
\label{di}
\dot{\eta} \in \mathrm{co}(F)(\eta)
\end{equation}
is an absolutely continuous function $\eta \in \cC(\dR_+,\dR^d)$ such that
$\dot{\eta}(t) \in \mathrm{co}(F)(\eta(t))$ for almost all $t \in \dR_+$.
We let $S^x \subset \cC(\dR_+,\dR^d)$ denote the set of solutions to \eqref{di}
with initial condition~$\eta(0)=x$. 

\begin{lem}\label{Sx}
 The set $S^x$  is a non empty compact connected set.
\end{lem}
 \begin{proof}
Follows from standard results  on differential inclusion, since the set-valued
map $\mathrm{co}(F)$ is  upper-semi continuous, bounded  with  non empty
compact convex images; see~\cite{AubCel84} for details.
\end{proof}


\begin{thm}\label{supportlaw}
If $X_0=x\in M$ then,  the support of the law of ${(X_t)}_{t \geq 0}$
equals $S^x$.
\end{thm}

\begin{proof}
Obviously, any path of $X$ is a solution of the differential inclusion~\eqref{di}.
Let $\eta \in S^x$, and $\varepsilon>0$. Set
\[
   G_t(x) =
   \BRA{v \in  \BRA{F^i(x),\ i \in E}\,:\,
     \langle v - \dot{\eta}(t), x - \eta(t) \rangle < \varepsilon}.
\]
Since $\dot{\eta}(t) \in \mathrm{co}(F)(\eta(t))$ almost surely, $G_t(x)$ is non
empty. Furthermore, $(t,x) \mapsto G_t(x)$ is uniformly bounded,
lower semicontinuous in $x,$ and measurable in $t.$
Hence, using a result by Papageorgiou \cite{Pap86},
there exists $\xi : \dR \to \dR^d$ absolutely continuous such that
$\xi(0) = x$ and $\dot{\xi}(t) \in G_t(\xi(t))$ almost surely. In particular,
\[
\frac{d}{dt} \|\xi(t)-\eta(t)\|^2=2\langle\dot{\xi}(t) -  \dot{\eta}(t),\xi(t)-\eta(t)\rangle
< 2 \varepsilon
\]
so that
\[
  \sup_{0 \leq t \leq T} \|\xi(t)-\eta(t)\|^2 \leq 2 \varepsilon T.
\]
Thus, without loss of generality, one can assume that $\eta$ is such that 
\[
\dot{\eta}(t) \in \bigcup_{i \in E} \BRA{ F^i(\eta(t))}
\]
for almost all $t \in \dR_+$. Set
\[
\forall i\in E, \qquad
\Omega_{i} = \BRA{  t \in [0,T]  \: :
                       \dot{\eta}(t) = F^{i}(\eta(t))}.
\]
Let $\mathcal{C}$ be the algebra consisting of finite unions of intervals
in $[0,T]$. Since the Borel $\sigma$-field over $[0,T]$ is generated by
$\mathcal{C}$, there exists, for all $i\in E$, $J_i \in \mathcal{C}$
such that, the set  $\{J_i \: : i\in E\}$ forms a partition of $[0,T]$, and for $i\in E$
\[
\lambda (\Omega_i \Delta J_i) \leq \varepsilon  
\]
where $\lambda$ stands for the Lebesgue measure over $[0,T]$ 
and $A \Delta B $ is the symmetric difference of $A$ and $B$. Hence,
there exist numbers $0 = t_0 < t_1 < \cdots < t_{N+1} = T$
and a map $i : k \mapsto i_k$ from $\{0,\ldots,N \}$ to $\{1,\ldots,n_0\}$, 
such that $(t_k,t_{k+1}) \subset J_{i_k}$.
Introduce $\eta_{x,\bi,\bu}$ given by Formula \eqref{def:etaiu}.
For all $t_k \leq t \leq t_{k+1}$,
\begin{align*}
\eta_{x,\bi,\bu}(t) - \eta(t) =&\eta_{x,\bi,\bu}(t_k) - \eta(t_k)\\
&    +  \int_{t_k}^t\!\PAR{F^{i_k}(\eta_{x,\bi,\bu}(s)) - F^{i_k}(\eta(s))}\,ds
  + \int_{t_k}^t\!\PAR{ F^{i_k}(\eta(s)) - \dot{\eta}(s)}\,ds.
\end{align*}
Hence, by Gronwall's lemma, we get that
\[
\|\eta_{x,\bi,\bu}(t) - \eta(t)\| \leq e^{K (t_{k+1}- t_k)} \PAR{ \NRM{\eta_{x,\bi,\bu}(t_k)-\eta(t_k)} + m_k}
\]
where $K$ is a Lipschitz constant for all the vector fields $(F^i)$ and
\[
m_k =
2C_{sp} \lambda([t_k,t_{k+1}] \setminus \Omega_{i_k} ).
\]
It then follows that, for all $k = 0, \ldots, N$ and
 $t_k \leq t \leq t_{k+1},$
\[
 \NRM{\eta_{x,\bi,\bu}(t) - \eta(t)} \leq
 \sum_{l = 0}^k e^{K(t_{k+1}-t_l)}m_l
 \leq e^{KT} \sum_{l = 0}^N m_l
 \leq e^{KT} \sum_{i = 1}^{\ABS{E}} \lambda(J_i \setminus \Omega_i)\leq e^{KT} \varepsilon.
\]
This, with Lemma~\ref{le:tube}, shows that $S^x$ is included in the support of 
the law of ${(X_t)}_{t\geq 0}$ and concludes the proof. 
\end{proof}
In the course of the proof, one has obtained the following result which is stated separately 
since it will be useful in the sequel. 
\begin{lem}\label{bang}
If $\eta:\ \dR_+\to\dR^d$ is  such that 
\[
\dot{\eta}(t) \in \bigcup_{i \in E} \BRA{ F^i(\eta(t))}
\]
for almost all $t \in \dR_+$, then, for any $\varepsilon>0$ and any $T>0$, there 
exists $(\bi,\bu)\in\TT_n$ (for some $n$) such that 
\[
 \NRM{\eta_{x,\bi,\bu}(t) - \eta(t)} \leq\varepsilon.
\] 
\end{lem}

\subsection{The accessible set}
\label{sec:acces}
In this section we define and study the  \emph{accessible set} of the process $X$ 
as the set of points that can be ``reached from everywhere'' by $X$ and show that 
long term behavior of $X$ is related to this set. 

\begin{defi}[Positive orbit and accessible set]\ 

For $(\bi,\bu)\in \TT_n$, let $\bphi^\bi_\bu$ be the ``composite flow'':
\begin{equation}
\label{defG}
{\bphi}^{\bi}_{\bu}(x) = \Phi_{u_{n}}^{i_{n-1}} \circ \ldots \circ \Phi^{i_0}_{u_1}(x).
\end{equation}
The \emph{positive} orbit of $x$ is the set
\[
\gamma^+(x) =
\BRA{\bphi^{\bi}_{\bu}(x) \: : (\bi, \bu) \in \bigcup_{n \in \mathbb{N}^*} \TT_n }.
\]
The \emph{accessible set} of $(X_t)$  is the  (possibly empty) compact
set $\Gamma \subset M$ defined as
\[
\Gamma = \bigcap_{x \in M} \overline{\gamma^+(x)}.
\]
\end{defi}

\begin{rem} If $y=\bphi^{\bi}_{\bu}(x)$ 
for some $(\bi, \bu)$ then $y$ is the limit of points of the form $\bphi^{\bj}_{\bv}(x)$ 
where $({\bj},{\bv})$ is adapted to $x$. It implies that
\begin{equation}\label{adapte}
\Gamma = \bigcap_{x \in M} \overline{\{\bphi^{\bi}_{\bu}(x) \: : (\bi, \bu)\ \text{ adapted to }x  \}}.
\end{equation}
\end{rem}

\begin{rem}
 The accessible set $\Gamma$ is called the set of  $D$-approachable points
 and is denoted by $L$ in \cite{bakhtin&hurt}.
\end{rem}

\subsubsection{Topological properties of the accessible set}
The differential inclusion (\ref{di}) induces a set-valued dynamical
system $\Psi = \{\Psi_t\}$ defined by
\[
\Psi_t(x) = \Psi(t,x) = \{ \eta(t) : \: \eta \in S^x\}
\]
 enjoying the following properties
 \begin{enumerate}[label={\bfseries(\roman*)}]
\item $\Psi_0(x) = \{x\}$,
\item $\Psi_{t+s}(x) = \Psi_t(\Psi_s(x))$ for all $t,s \geq 0$,
\item $y \in \Psi_t(x) \Rightarrow x \in \Psi_{-t}(y)$.
\end{enumerate}
For subsets $I \subset \dR$ and $A\subset \dR^d$ we set
\[
\Psi(I,A) = \bigcup_{(t,x) \in I \times A} \Psi_t(x).
\]
A set $A \subset \dR^d$ is called \emph{strongly} positively invariant
under $\Psi$ if $\Psi_t(A) \subset A$ for $t \geq 0.$
It is called \emph{invariant} if for all $x \in A$ there exists $\eta \in S^x_\dR$
such that $\eta(\dR) \subset A$. Given $x \in \dR^d,$ the \emph{(omega) limit set}
of $x$ under $\Psi$ is defined as
\[
\omega_{\Psi}(x) = \bigcap_{t \geq 0} \overline{\Psi_{[t,\infty[}(x)}.
\]

\begin{lem}
\label{limsetprop}
The set $\omega_{\Psi}(x)$ is compact connected invariant and
strongly positively invariant under $\Psi$.
\end{lem}
\begin{proof} It is not hard to deduce the first three properties from
Lemma~\ref{Sx}. For the last one, let $p \in \omega_{\Psi}(x), \, s > 0$
and $q \in \Psi_s(p)$. By Lemma~\ref{bang}, for all $\eps > 0$, there exists
$n \in \dN$ and $(\bi,\bu) \in \TT_n$ such that $\NRM{\bphi^{\bi}_{\bu}(p)-q} < \eps$.
Continuity of $\Phi^{\bi}_{\bu}$ makes the set
\[
\cW_\varepsilon = \{z \in M : \: \NRM{\bphi^{\bi}_{\bu}(z)- q} < \eps\}
\]
an open neighborhood of $p$. Hence $\cW_\varepsilon \cap \Psi_{[t,\infty[}(x) \neq \emptyset$
for all $t > 0$. This proves that the distance between the sets $\BRA{q}$ and $\Psi_{[t,\infty[}(x)$ 
is smaller than $\eps$. Since $\eps$ is arbitrary, $q$ belongs to $\overline{\Psi_{[t,\infty[}(x)}$.
\end{proof}

\begin{rem}
For a general differential inclusion with an upper semi-continuous bounded
right-hand side with  compact convex values, the omega limit set of a point
is not (in general)  strongly positively invariant,  see e.g.~\cite{BHS1}.
\end{rem}

\begin{prop}[Properties of the accessible set]
\label{bangbang}
The set $\Gamma$ satisfies the following:
\begin{enumerate}[label={\bfseries(\roman*)}]
  \item \label{intersection}
 $\Gamma = \bigcap_{x \in M} \omega_{\Psi}(x)$,
 \item \label{gammaP}
   $\Gamma = \omega_{\Psi}(p)$ for all $p \in \Gamma$,
 \item \label{compacity}
   $\Gamma$ is compact, connected, 
   invariant and strongly positively invariant under $\Psi$,
 \item \label{interior}
  either $\Gamma$ has empty interior or  its interior is dense in $\Gamma$.
\end{enumerate}
\end{prop}

\begin{proof}
  \ref{intersection}~: Let $x \in M$ and $y \in \Psi_t(x).$ Then
  $\gamma^+(y) \subset \Psi_{[t,\infty]}(x).$ Hence
  $\Gamma \subset \overline{\Psi_{[t,\infty[}(x)}$ for all $x$. This proves
  that $\Gamma \subset \bigcap_{x \in M} \omega_{\Psi}(x)$. Conversely
  let $p \in \bigcap_{x \in M} \omega_{\Psi}(x).$ Then, for all $t > 0$ and
  $x \in M$, $p \in \overline{\Psi_{[t,\infty[}(x)} \subset \overline{\gamma^+(x)}$
  where the latter inclusion follows from  Lemma~\ref{bang}. This proves
  the converse inclusion.

  \ref{gammaP}~: By Lemma \ref{bang}, $\omega_{\Psi}(p) \subset \Gamma$.
  The converse inequality follows from \ref{intersection}.

  \ref{compacity}~: This follows from~\ref{gammaP} and Lemma~\ref{limsetprop}.

  \ref{interior}~: Suppose $\mathrm{int}(\Gamma) \neq \emptyset$.
  Then there exists an open set $\cU \subset \Gamma$ and   $\bigcup_{\bt,\bi} \bphi_{\bt}^{\bi}(\cU)$
  is an open subset of $\Gamma$ dense in $\Gamma.$
\end{proof}

An equilibrium $p$ for the flow $\Phi^1$ is a point in $M$ such that 
$F^1(p)=0$. It is called an \emph{attracting equilibrium} if there exists a 
neighborhood $\cU$ of $p$ such that
\[
\lim_{t \rightarrow \infty} \|\Phi_t^1(x)-p\| = 0
\]
uniformly in $x \in \cU$. In this case, the \emph{basin of attraction} of $p$
is the open set
\[
\cB(p) = \BRA{x \in \dR^d: \: \lim_{t \rightarrow \infty} \|\Phi_t^1(x)-p\| = 0}.
\]

\begin{prop}[Case of an attracting equilibrium]
\label{attractequil}
Suppose the flow $\Phi^1$ has an attracting equilibrium  $p$ with
basin of attraction $\cB(p)$ that intersects all orbits:  for all $x \in M \setminus \cB(p)$,
$\gamma^+(x) \cap \cB(p) \neq \emptyset$.
Then
\begin{enumerate}[label={\bfseries(\roman*)}]
  \item \label{gammaP2}
    $\Gamma = \overline{\gamma^+(p)}$,
  \item \label{contractile} If furthermore $\Gamma \subset \cB(p)$, then
    $\Gamma$ is contractible. In particular, it is simply connected.
\end{enumerate}
\end{prop}

\begin{proof} The proof of~\ref{gammaP2} is left to the reader. To
prove~\ref{contractile},  let $h : [0,1] \times \Gamma \rightarrow\Gamma$
be defined by
\[
h(t,x) =
\begin{cases}
\Phi^1_{-\log(1-t)}(x)&\text{ if } t < 1,\\
p&\text{ if }t=1.
\end{cases}
\]
It is easily seen that $h$ is continuous. Hence the result.
\end{proof}

\subsubsection{The accessible set and recurrence properties}

In this section, we link the accessibility (which is a deterministic notion) to some recurrence 
properties for the embedded chain $\tilde Z$ and the continuous time process $Z$.    
\begin{prop}[Returns near $\Gamma$ --- discrete case]
  \label{Harris1}
Assume that $\Gamma \neq \emptyset.$  Let $p \in \Gamma$ and  $\cU$
be a neighborhood of $p.$ There exist $m \in \mathbb{N}$ and $\delta > 0$
such that for  all $i,j \in E$ and $x \in M$
\[
  \prb[x,i]{\tilde{Z}_m \in \cU \times \{j\}} \geq \delta.
\]
In particular,
\[
  \prb[x,i]{\exists l \in \dN,  \: \tilde{Z}_l \in \cU \times \{j\}} = 1.
\]
\end{prop}
Note that, in the previous proposition, the same discrete time $m$ works 
for all $x\in M$. In the continuous time framework, one common time $t$ 
does not suffice in general; however one can prove
a similar statement if one allows a finite number of times:
\begin{prop}[Returns near $\Gamma$ --- continuous case]
\label{Harriscont1}
Assume that $\Gamma \neq \emptyset$.
 Let $p \in \Gamma$ and  $\cU$ be a neighborhood of $p.$ There exist
 $N\in \mathbb{N}$, 
 a finite open covering $\cO^1, \ldots \cO^N$ of $M$, 
  $N$ times $ t_1, \ldots, t_N > 0 $ and $\delta > 0$ such that, for
 all $i,j \in E$ and $x \in \cO^k$,
\[
  \prb[x,i]{ Z_{t_k} \in \cU \times \{j\} }\geq \delta.
\]
\end{prop}
Since $\Gamma$ is positively invariant under each flow $\Phi^i$ we
deduce from Proposition~\ref{Harriscont1} the following result.
\begin{cor}
\label{Harriseasycor}
Assume $\Gamma$ has non empty interior. Then
\[
  \prb[x,i]{ \exists t_0 \geq 0,  \forall t \geq t_0,  Z_{t} \in \Gamma \times E} = 1.
\]
\end{cor}

Propositions~\ref{Harris1} and \ref{Harriscont1} are direct consequences of 
Lemma~\ref{le:tube} and of the following technical result. 
\begin{lem}\label{triv2}
Assume $\Gamma \neq \emptyset.$
Let $p \in \Gamma,$  $\cU$ be a neighborhood of $p$ and $i,j \in E$.
There exist $m \in \mathbb{N}^*, \varepsilon, \beta > 0,$ finite sequences
$(\bi^1,\bu^1) \ldots, (\bi^N,\bu^N) \in \TT^{ij}_m$ and an  open covering
$\cO^1, \ldots, \cO^N$ of $M$ (i.e. $M = \cO^1 \cup \ldots \cup \cO^N$ )
such that for all $x \in M$ and $\mathbf{\tau} \in \Rp^m$:
\[
x \in \cO^k \text{ and } \NRM{\tau - \bu^k} \leq \varepsilon
\quad\Longrightarrow\quad
 \bphi_{\tau}^{\bi^k}(x) \in \cU \text{ and } p(x,\bi^k,\tau) \geq \beta.
\]
Furthermore, $m, \varepsilon$ and $\beta$ can be chosen independent of $i,j \in E$.
\end{lem}
\begin{proof} 
Fix $i$ and $j$, and let $\cV$ be a neighborhood of $p$ with closure $\overline{\cV} \subset \cU$.
For all $\beta>0$, define the open sets
\[ 
  \cO(\bi,\bu,\beta) = \{x \in M : \: \bphi_{\bu}^{\bi}(x) \in \cV, p(x,\vtr{i}, \vtr{u}) > \beta \}.
\]
By \eqref{adapte}, one has 
\begin{equation}
  \label{eq:bigInclusion}
  M = \bigcup_n \PAR{\bigcup_{\beta>0}\bigcup_{(\bi,\bu) \in \TT^{ij}_{n}}\cO(\bi,\bu,\beta)}
\end{equation}
Now, since one can add ``false'' jumps ($i$ does not change and $u$ equals 0) 
to $(\vtr{i},\vtr{u})$ without changing $\bphi_{\vtr{u}}^{\vtr{i}}(x)$, 
\[
  \forall (\vtr{i},\vtr{u}) \in \TT^{ij}_n, \forall n'\geq n, \forall \beta, 
  \;
  \exists \beta'>0 , (\vtr{i'},\vtr{u'})\in \TT^{ij}_{n'}
  \quad \text{such that} \quad 
  \cO(\vtr{i},\vtr{u}, \beta) \subset \cO(\vtr{i'},\vtr{u'}, \beta')
\]
(just add $n'-n$ false jumps at the beginning and let $\beta' = \beta (\inf_M Q(x,i,i))^{n' - n}$). 
Therefore the union over $n$ in \eqref{eq:bigInclusion} is increasing: by compactness, there exists $m$ such that
\[
  M \subset
\bigcup_{\beta>0}\bigcup_{(\bi,\bu) \in \TT^{ij}_{m}}\cO(\bi,\bu,\beta).
\]
Note that by monotonicity, $m$ can be chosen uniformly over $i$ and $j$.
The union in $\beta$ increases as $\beta$ decreases, so by compactness again there
exists $\beta_0$ (independent of $i$, $j$) such that
\[
  M \subset \bigcup_{(\bi,\bu) \in \TT^{ij}_{m}}\cO(\bi,\bu,\beta_0).
\]
A third invocation of compactness shows that for some finite $N$, 
\[ M \subset \bigcup_{k=1}^N \cO_k,\]
where $\cO^k = \cO(\vtr{i}^k, \vtr{u}^k, \beta_0)$ for some
$(\vtr{i}^k, \vtr{u}^k)$ in $\TT^{ij}_m$. Since $\overline{\cV} \subset \cU$, 
the distance between $\cV$ and $\cU^c$ is positive (once more by compactness).
Choosing $\eps$ small enough therefore  guarantees
that  for $x\in\cO^k$ and $\NRM{\vtr{v} - \vtr{u}^k}<\varepsilon$, $\bphi_\vtr{i}^\vtr{v}(x) \in \cU$. 
This concludes the proof. 
\end{proof}


\subsection{Support of invariant probabilities}
\label{sec:stable}

The following proposition relates $\Gamma$ to the support of the invariant measures of 
$\tZ$ or $Z$. We state and prove the result for $\tZ$ and
rely on Lemma~\ref{samesupport} for $Z$.

\begin{prop}[Accessible set and invariant measures]
  \label{thsupport}
  \ 

\begin{enumerate}[label={\bfseries(\roman*)}]
  \item \label{suppI}
If $\Gamma \neq \emptyset$ then
 $\Gamma \times E\subset \supp(\mu)$ for all
 $\mu \in  \mathcal{\tP}_{inv}$ and
 there exists $\mu \in \mathcal{\tP}_{inv}$ 
 such that $\supp(\mu) = \Gamma \times E.$
\item \label{suppII}
If $\Gamma$ has non empty interior, then $\Gamma \times E = \supp(\mu)$
for all $\mu \in \mathcal{\tP}_{inv}$.
\item \label{suppIII}
  Suppose that $\tZ$ is stable with invariant probability $\pi$,
  then  $\supp(\pi) = \Gamma \times E$.
\end{enumerate}
\end{prop}

\begin{proof}
\ref{suppI} follows from Proposition~\ref{Harris1}. Also,  since $\Gamma$ is
strongly positively invariant, there are invariant measures supported by
$\Gamma \times E$.
  \ref{suppII} follows from \ref{suppI} and Corollary~\ref{Harriseasycor}.
  To prove~\ref{suppIII},  let $(p,i) \in \supp(\pi)$. Let $\cU, \cV$ be
  open neighborhoods of $p$ with $\overline{\cU} \subset \cV$ compact. Let
  $0 \leq f \leq 1$ be a continuous function on $M$ which is $1$ on $\cU$ and $0$
  outside $\cV$ and let $\tilde{f}(x,j) = f(x)\delta_{j,i}$. Suppose $Z_0 = (x,j)$.
Then, with probability one,
\begin{align*}
  \liminf_{n \rightarrow \infty}
  \frac{1}{n} \sharp\BRA{1 \leq k \leq n \: : \tZ_k \in \cV \times \{i\} }   
 \geq  &\lim_{n \rightarrow \infty} \frac{1}{n} \sum_{k = 1}^n \tilde f(\tZ_k) 
 \geq  \pi(\cU \times \{i\}) >0.
 \end{align*}
 Hence $(\tZ_n)$ visits infinitely often $\cU \times \{i\}.$ In particular,
 $p \in \gamma^+(x).$ This proves that $\supp(\pi) \subset \Gamma \times E$.
 The converse statement follows from~\ref{suppI}.
\end{proof}
\begin{rem}
The example given in Section~\ref{sec:ovidiu} shows that
the inclusion $\Gamma \times E \subset \supp(\mu)$ may be strict when
$\Gamma$ has empty interior. On the other hand, the condition that
$\Gamma$ has non empty interior is not  sufficient to ensure uniqueness
of the invariant probability since there exist  smooth minimal flows
that are not uniquely ergodic. An example of such a flow can be constructed
on a $3$-manifold by taking the suspension of an analytic minimal non
uniquely ergodic diffeomorphism of the torus constructed by Furstenberg
in~\cite{Furst} (see also~\cite{Mane}). As shown in \cite{bakhtin&hurt} (see also
Section~\ref{sec:reg-and-erg}) a sufficient condition to ensure uniqueness
of the invariant probability is that the vector fields verify a Hörmander
bracket property at some point in $\Gamma$.
\end{rem}

\section{Absolute continuity and ergodicity}
\label{sec:reg-and-erg}

\subsection{Absolute continuity of the law of the processes}


Let $x_0$ be a point in $M$. The image of $u\mapsto \Phi_u(x_0)$ is a curve; one 
might expect that, if $i\neq j$, the image of $(s,t)\mapsto \Phi_t^{j}(\Phi_s^i(x_0))$ 
should be a surface. Going on composing the flow functions in this way, one might 
fill some neighbourhood of a point in  $\dR^d$.

Recall that, if $\vtr{i}$ is a sequence of indices
$\vtr{i} = (i_0, \ldots i_m)$ and $\vtr{u}$ is a sequence of times $\vtr{u} =
(u_1, \ldots u_m)$, $\PhiIU:M\to M$ is the composite map
defined by \eqref{defG}.

Suppose that for some $x_0$ the map $\vtr{v} \mapsto \PhiIV(x_0)$ is a submersion 
at  $\vtr{u}$. Then the image of the Lebesgue measure on a neighbourhood of $\vtr{u}$ 
is a measure on a neighbourhood of $\PhiIU(x_0)$ equivalent to the Lebesgue measure. 
If the jump rates $\lambda_i$ are constant functions, the probability 
$ \prb[x_0,i_0]{\tilde Z_m \in \cdot\times\BRA{i_m}}$ is just the image by 
$\vtr{u}\in\dR^m \mapsto \PhiIU(x_0)\in \dR^d$ 
of the product of exponential laws on $\dR^m$.  
We get that there exist $\cU_0$ 
a neighborhood of $x_0$, $\cV_0$ a neighborhood of $\PhiIU(x_0)$, and 
a constant $c>0$ such that:
\[
     \forall x\in \cU_0, \quad \prb[x,i_0]{\tilde Z_m\in \cdot\times\BRA{i_m}}
     \geq c \leb{d}\PAR{ \cdot \cap \cV_0}.
\]

Let us fix a $t>0$ and consider now the function 
$\vtr{v} \mapsto \Phi^{i_m}_{t-v_1-\ldots- v_k}(\PhiIV(x_0))$ defined on 
$v_1+\ldots +v_m<t$. For the same reason, if this function is a submersion at $\bu$, 
then the law of $X_t$ has an absolutely continuous part with respect to 
$\lambda_{\dR^d}$.  

The two following results state a stronger result (with a local uniformity with respect to 
initial and final positions) both for the embedded chain and the continuous time process. 
\begin{thm}[Absolute continuity --- discrete case]
  \label{thm=globalRegJumps}
  Let $x_0$ and $y$ be two points in $M$ and a sequence $(\vtr{i},\vtr{u})$ such that 
  $\PhiIU(x_0)=y$. If $\vtr{v} \mapsto \bphi^{\vtr{i}}_{\vtr{v}}(x_0)$ is a submersion at $\vtr{u}$, 
  then there exist $\mathcal{U}_0$ a neighborhood of $x_0$, $\mathcal{V}$ a 
  neighborhood of $y$, an integer $m$ and a constant $c>0$ such that 
   \begin{equation}
     \forall x\in \cU_0,\ \forall i,j\in E, \quad \prb[x,i]{\tilde Z_{m} \in \cdot\times \BRA{j}}
     \geq c \leb{d}\PAR{ \cdot \cap \cV}.
     \label{eq=regAtJumpTimes}
   \end{equation}
 \end{thm}

 \begin{thm}[Absolute continuity --- continuous case]
   \label{thm=globalRegFixed}
     Let $x_0$ and $y$ be two points in $M$ and a sequence $(\vtr{i},\vtr{u})\in \dT_m$ 
     such that $\PhiIU(x_0)=y$. If 
  $\vtr{v} \mapsto \bphi^{i_m}_{s-(v_1+\cdots+v_m)}\bphi^{\vtr{i}}_{\vtr{v}}(x_0)$ 
  is a submersion at $\vtr{u}$ for some $s>u_1+\cdots+u_m$, 
  then for all $t_0>u_1+\cdots+u_m$, there exist
  $\mathcal{U}_0$ a neighborhood of $x_0$, $\mathcal{V}$ a neighborhood of $y$
  and two constants $c,\varepsilon>0$ such that 
   \begin{equation}
     \forall x\in \cU_0,\ \forall i,j\in E,\ \forall t\in[t_0,t_0+\varepsilon],
     \quad \prb[x,i]{Z_t \in \cdot\times \BRA{j}}
     \geq c \leb{d}\PAR{ \cdot \cap \cV}.
    \label{eq=regAtFixedTimes}
   \end{equation}
\end{thm}

The proofs of these results are postponed to Section~\ref{absolute}. 

Unfortunately, the hypotheses of these two theorems  are not easy to check, since one needs
to ``solve'' the flows. However, they translate to two very nice local conditions.
To write down these conditions, we need a bit of additional notation. Let
$\mathcal{F}_0$ the collection of vector fields $\BRA{F^i \: : i \in E}$. Let
$\mathcal{F}_k = \mathcal{F}_{k-1} \cup\{ [F^i,V], V\in \mathcal{F}_{k-1}\}$ (where 
$[F,G]$ stands for the Lie bracket of two vector fields $F$ and $G$)
and $\mathcal{F}_k(x)$ the vector space (included in $\dR^d$) spanned
by $\{V(x), V\in\mathcal{F}_k\}$.

Similarly, starting from $\mathcal{G}_0 = \{F^i - F^j, i\neq j\}$, we
define $\mathcal{G}_k$ by taking Lie brackets with the vector fields $\BRA{F^i \: : i \in E}$, 
and $\mathcal{G}_k(x)$ the corresponding subspace of $\dR^d$.

\begin{defi}
  \label{defi:bracketCondition}
  We say that the \emph{weak bracket condition} (resp. \emph{strong bracket condition}) 
  is satisfied at $x\in M$ if there exists $k$ such that $\mathcal{F}_k(x) = \dR^d$
  (resp. $\mathcal{G}_k(x) = \dR^d$). 
\end{defi}
These two conditions are called A (for the stronger) and B (for the weaker)
in \cite{bakhtin&hurt}.
Since $\mathcal{G}_k(x)$ is a subspace of $\mathcal{F}_k(x)$, the
strong condition implies the weak one. The converse is false, a
counter-example is given below in Section~\ref{sec:torus}.

\begin{thm}
  \label{th:hormander}
  If the weak (resp.\ strong) bracket condition holds at $x_0\in M$, then the conclusion of 
  Theorem~\ref{thm=globalRegJumps} (resp. Theorem~\ref{thm=globalRegFixed}) holds.
\end{thm}

This theorem is a version of Theorem~2 from~\cite{bakhtin&hurt} with an additional uniformity 
with respect to the initial point and the time $t$. Thanks to Theorems~4 and~5 in \cite{bakhtin&hurt}, 
one can deduce the hypotheses of Theorems~\ref{thm=globalRegJumps} 
and \ref{thm=globalRegFixed} from the bracket conditions. 
The proofs in \cite{bakhtin&hurt} are elegant but 
non-constructive; we give a more explicit proof in Section~\ref{alamain}.

\subsection{Ergodicity}
\label{sec:erg}

\subsubsection{The embedded chain}

\begin{thm}
  [Convergence in total variation --- discrete case]
  \label{THWB}
Suppose there exists $p \in \Gamma$ at which the weak bracket condition holds.
Then the chain $\tilde{Z}$ admits a unique invariant probability $\tilde{\pi}$,
absolutely continuous with respect to the Lebesgue measure~$\lambda_{M \times E}$ 
on $M\times E$. Moreover, there exist 
two constants $c>1$ and $\rho\in (0,1)$ such that, for any $n\in\dN$, 
\[
  \TV{ \prb{\tZ_n \in \cdot} - {\tilde \pi}} \leq c \rho^n
\]
where $\TV{ \cdot }$ stands for the total variation norm.
\end{thm}

\begin{proof}
By Proposition~\ref{Harris1} and Theorem~\ref{th:hormander}, 
there exist a neighborhood $\cU_0$  of $p$, integers $m$ and $K$, a measure $\psi$ on 
$M \times E$ absolutely continuous with respect to $\lambda_{M \times E}$ and $c > 0$ 
such that, setting $A = \cU_0 \times E$,
\begin{enumerate}[label={\bfseries(\roman*)}]
  \item \label{it:paf} $\prb[x,i]{\tZ_m \in A} \geq \delta$ for all $(x,i) \in M \times E$,
  \item \label{it:paf2} $\prb[x,i]{\tZ_K \in \cdot} \geq c \psi(\cdot)$ for all $(x,i) \in A$.
\end{enumerate}
These two properties make $\tZ$ an Harris chain  with recurrence set $A$ 
 (see e.g.~\cite[Section~7.4]{durrett}). It is recurrent, aperiodic (by~\ref{it:paf}) 
 and the first hitting time of 
 $A$ has geometric tail (by~\ref{it:paf} again). Therefore,  by usual arguments, 
 two copies of $\tZ$ may be coupled in a time $T$ that has geometric tail; 
 this implies the exponential convergence in total variation toward its 
 (necessarily unique) invariant probability  $\tilde{\pi}$ (see e.g. the proof 
 of Theorem~4.10 in \cite[Section~7.4]{durrett} or \cite[Section~I.3]{Lindvall}, 
  for details).

Finally, observe that, by~\ref{it:paf2} and Theorem~\ref{th:hormander}, $\tilde{\pi} \geq  \delta c \psi$. 
Therefore $\tilde{\pi}_{ac}$ (the absolutely continuous part of $\tilde{\pi}$ with respect 
to $\lambda_{M \times E}$) is non zero. Then, Proposition~\ref{typepur} 
ensures that $\tilde{\pi}$ is absolutely continuous with respect to $\lambda_{M \times E}$.
\end{proof}

As pointed out in \cite[Theorem~1]{bakhtin&hurt} (see also Theorem~\ref{homeo}), 
 under the hypotheses of Theorem~\ref{THWB}, ${(Z_t)}_{t\geq 0}$ 
admits a unique invariant probability measure $\pi$, and $\pi$ is absolutely 
continuous with respect to $\lambda_{M \times E}$. Moreover, under the hypotheses 
of Theorem~\ref{THWB},  with probability one
\[ 
\lim_{n \rightarrow \infty} \tilde{\Pi}_n = \tilde{\pi} 
\quad\text{and}\quad
\lim_{t \rightarrow \infty} \Pi_t = \tilde{\pi} \tK.
\]

Under the strong bracket assumption, we prove in the next section
that the distribution of $Z_t$ itself converges, and not only its empirical
measure.

\subsubsection{The continuous time process}

\begin{thm}[Convergence in total variation --- continuous case]
  \label{th:cvgContinue}
  Suppose that there is a point $p \in \Gamma$ at which the strong bracket
  condition is satisfied. Let $\pi$ be the unique invariant probability measure of $Z$.
  Then there exist two constants $c>1$ and $\alpha>0$ such that, for any $t\geq 0$, 
  \begin{equation}
    \TV{\prb{Z_t \in \cdot} - \pi} \leq c e^{-\alpha t}.
    \label{eq=convergenceContinue}
  \end{equation}
\end{thm}

\begin{proof}
The proof consists in showing that there exist a neighborhood $\cU$ of $p$,
 $t>0$, $j\in E$ and $\beta>0$ such that for all $x\in M$ and $i\in E$,
  \begin{equation}
    \label{eq=uniformReturn}
    \prb[x,i]{Z_t \in \cU\times\BRA{j}} \geq \beta.
  \end{equation}
  This property and \eqref{eq=regAtFixedTimes} ensure that 
two processes starting from anywhere can be coupled in some time $t$ 
with positive probability. This, combined with Theorem~\ref{th:hormander}, 
implies~\eqref{eq=convergenceContinue} by the usual coupling argument (see \cite{Lindvall}).

Theorem~\ref{th:hormander} gives us two open sets $\cU_0$, $\cV$ (with $p\in\cU_0$), 
a time $s_0$  and $\varepsilon>0$ such that
   \begin{equation} \label{eq=fixedTimesAgain}
     \forall x\in \mathcal{U}_0, \forall t \in [s_0, s_0 + \varepsilon], \forall i,  \forall j,  \quad
     \prb[x,i]{Z_t \in \cdot \times \{j\}} \geq c \lambda_{\dR^d}\PAR{ \cdot \cap \cV}.
   \end{equation}
Moreover, we have seen in Proposition~\eqref{Harriscont1} that there exist
$m$ times $ t_1, \ldots, t_m > 0 $, $m$ open sets $\cO^1,\ldots,\cO^m$ (covering $M$) 
and $\delta > 0$ such that, for all $i,j \in E$ and $x \in \cO^k$,
\begin{equation}\label{plusieurs}
  \prb[x,i]{Z_{t_k} \in \cU_0 \times \{j\}} \geq \delta.
\end{equation}

We can suppose that $ \cV $ is included in one of the $\cO^k$ (just shrink $\cV$ if necessary). 
Then, there exists $s\in \{t_1, \ldots, t_m\} $ such that
    \begin{align} \label{eq=blah}
  \forall y\in \cV, \forall i,\forall j,\quad
  \prb[y,i]{Z_{s} \in  \cU_0 \times \{j\}} &\geq \beta_1.
\end{align}
From (\ref{eq=fixedTimesAgain}) and (\ref{eq=blah}), we get a time $t_0(=s_0+s)$ 
and $\beta_2>0$, such that
\[
     \forall x\in \mathcal{U}_0, \forall t \in [t_0, t_0 + \varepsilon], \forall i,  \forall j, \quad
     \prb[x,i]{Z_t \in \cU_0 \times \{j\} } \geq \beta_2.
\]
The Markov property then gives that, for every integer $n$, one has
 \begin{equation}
     \label{mou3}
     \forall x\in \mathcal{U}_0, \forall t \in [nt_0, nt_0 + n\varepsilon], \forall i, \forall j,  \quad
     \prb[x,i]{Z_{t} \in \cU_0  \times \{j\}} \geq \beta_2^n.
   \end{equation}
We can suppose that $t_1$ is the largest of the $m$ times ${(t_k)}_{1\leq k\leq m}$. 
Let $n$ be such that for every $k$ between $1$ and $m$ there exists $v_k\in [0,n\varepsilon]$ 
such that $t_1=t_k+v_k$.  Take such numbers ${(v_k)}_{1\leq k\leq m}$. If $x$ belongs to $\cO^k$, 
by (\ref{plusieurs}) and (\ref{mou3}), for all $i$ and $j$,
 \[
   \prb[x,i]{Z_{t_1+nt_0} \in \cU_0  \times \{j\}}
   = \prb[x,i]{Z_{t_k+nt_0+v_k} \in \cU_0  \times \{j\}}\geq \delta\beta_2^n.
 \]
 This concludes the proof.
\end{proof}

\section{Examples}
\label{sec:elementary}

\subsection{On the torus}
\label{sec:torus}
Consider the system defined on the torus $\dT^d = \dR^d/\dZ^d$ by
the constant vector fields $F^i = e^i$, where $(e_1, \ldots e_d)$ is the
standard basis on $\dR^d$. Then, as argued in \cite{bakhtin&hurt}, the
weak bracket condition holds everywhere, and the strong condition
does not hold. Therefore the chain $\tZ$ is ergodic and converges
exponentially fast, the empirical means of $(\tZ_n)$ and $(Z_t)$ converge, but 
the law of $Z_t$ is singular with respect to the invariant measure for any $t>0$ 
provided it is true for $t=0$. 

\subsection{Two planar linear flows}\label{ex:linear}
Let $A$ be a $2 \times 2$ real matrix whose eigenvalues
$\eta_1,\eta_2$ have negative real parts. Set  $E= \{0,1\}$ and
consider the process defined on $\dR^2 \times E$ by
\[
F^0(x) = Ax \quad\text{and}\quad F^1(x) = A(x-a)
\]
for some $a \in \dR^2$. The associated flows are
$\Phi^0_t(x) = e^{tA} x$ and $\Phi^1_t(x) = e^{tA}(x-a) + a$. Each flow admits a 
unique equilibrium (which is attracting): 0 and $a$ respectively. 

First note that, by using the  Jordan decomposition  of $A,$ it is possible
to find a scalar product $\langle\, \cdot \,\rangle$ on $\dR^2$ (depending on $A$)
and some number $0 < \alpha \leq
\min(- \mathrm{Re}(\eta_1), - \mathrm{Re}(\eta_2))$ such that
$\langle Ax,x\rangle \leq - \alpha \langle x,x\rangle$.  Therefore
\[
\langle A(x-a),x\rangle
\leq - \alpha \langle x,x\rangle - \langle Aa,x\rangle \leq \|x\|(-\alpha \|x\| + \|Aa\|).
\]
This shows that, for $R > \|Aa\|/\alpha$, the ball
$M = \{x \in \dR^2,  \NRM{x} \leq R\}$ is positively invariant by
$\Phi^0$ and $\Phi^1.$ Moreover every solution to the differential inclusion
induced by $\{F^0, F^1\}$ eventually enters $M$. In particular $M \times E$
is an absorbing set for the process $Z$.

Another remark that will be useful in our analysis is that
\[
\det(F^0(x), F^1(x) ) = \det(A) \det(a,x),
\]
so that
\begin{equation}
\label{detAx}
\det(F^0(x), F^1(x) ) > 0 \; \text{ (resp.\ $=0$)}
\Leftrightarrow \det(a,x) > 0 \; \text{ (resp.\ $=0$)}.
\end{equation}
\subsubsection*{Case 1: $a$ is an eigenvector}\label{ex=degenerate}
If $a$ is an eigenvector of $A$, then the line $\dR a$ is invariant by both flows, so that
\[
\Gamma = \overline{\gamma^+(0)} =  [0,a]
\]
and there is a unique invariant probability $\pi$ (and its support is $\Gamma$ 
by Proposition~\ref{thsupport}).
Indeed, it is easily seen that $\Gamma$ is an attractor for the
set-valued dynamics induced by $F^0$ and $F^1.$ Therefore the support
of every invariant measure equals $\Gamma$. If we consider the
process restricted to $\Gamma$, it becomes one-dimensional
and the strong bracket condition holds, proving uniqueness.

\begin{rem}
If $X_0 \not \in \dR a$, $X$
will never reach $\Gamma$.  As a consequence, the law of $X_t$ and $\pi$
 are singular for any $t\geq 0$. In particular, their total variation distance
is constant, equal to $1$. Note also that, the strong bracket condition being 
satisfied everywhere except on  $\dR a$, the law of $X_t$ at any positive finite 
time has a non trivial absolutely continuous part. 
\end{rem}
\begin{rem}
\label{rem=unidim}
Consider the following example: $A=-I$, $a=(1,0)$ and $\dR a$ is identified
to $\dR$. If the jump rates are constant and equal to $\lambda$, it is easy
to check (see \cite{KB,RMC}) that the invariant measure $\mu$ on
$[0,1]\times \{0,1\}$ is given by:
\[
\mu = \frac{1}{2} \left( \mu_0 \times \delta_0 +  \mu_1 \times \delta_1\right),
\]
where $\mu_0$ and $\mu_1$ are Beta laws on $[0,1]$:
\begin{align*}
  \mu_0(dx) &= C_\lambda x^{\lambda-1}(1-x)^\lambda\,dx, \\
  \mu_1(dx) &= C_\lambda x^{\lambda}  (1-x)^{\lambda-1}\,dx.
\end{align*}
In particular, this example shows that the density of the invariant
measure (with respect to the Lebesgue measure) may be unbounded:
when the jump rate $\lambda$ is smaller than $1$, the densities blow
up at $0$ and $1$.
\end{rem}

\subsubsection*{Case 2: Eigenvalues are reals and $a$ is not an eigenvector}
Suppose that the two eigenvalues $\eta_1$ and $\eta_2$ of $A$ are negative real numbers
and that $a$ is not an eigenvector of $A$.

Let $\gamma_0 = \{\Phi^0_t(a), \, t \geq 0\}$ and 
$\gamma_1 = \{\Phi^1_t(0), \, t \geq 0\}$. Note that $\gamma_1$ and
$\gamma_0$ are image of  each other by the transformation $T(x) = a - x.$
The curve $\gamma_0$ (respectively $\gamma_1$)  crosses the line
$\dR a$ only at point $a$ (respectively $0$). Otherwise, the trajectory
$t \mapsto \Phi^0_t(a)$ would have to cross the line
$\mathrm{Ker}(A-\lambda_1 I)$  which is invariant. This makes the curve
$\gamma = \gamma_0 \cup \gamma_1$ a simple closed curve
in $\dR^2$ crossing $\dR a$ at $0$ and $a$. By Jordan curve Theorem,
$\dR^2 \setminus \gamma = \cB \cup \cU$ where $\cB$ is a  bounded component
and $\cU$  an unbounded one. We claim that
\[
\Gamma = \overline{\cB}.
\]
To prove this claim, observe that thanks to~\eqref{detAx},
$F^0$ and $F^1$ both point inward $\cB$ at every point of $\gamma$. This makes
$\overline{\cB}$ positively invariant by $\Phi^0$ and $\Phi^1$.
Thus $\Gamma \subset \overline{\cB}$. Conversely,
$\gamma  \subset \Gamma$ (because $0$ and $a$ are accessible
from everywhere). If $x \in \cB$ there exists $s >0$ such that
$\Phi^0_{-s}(x) \in \gamma$ (because
$\lim_{t \rightarrow - \infty} \ABS{\Phi^0_t(x)} = +\infty)$ and necessarily
$\Phi^0_{-s}(x) \in \gamma_1.$ This proves that $x \in \gamma^+(0).$
 Finally note that the strong bracket condition is verified in
 $\Gamma \setminus  \dR a$, proving uniqueness and absolute continuity
 of the invariant probability.

\begin{rem}
Note that if the jump rates are small,
the situation is similar to the one described in Remark~\ref{rem=unidim}: 
the process spends a large amount of time near the attractive points,
and the density is unbounded at these points. By the way, this is also the case on 
the boundary of $\Gamma$.
\end{rem}

\subsubsection*{Case 3: Eigenvalues are complex conjugate}
Suppose now that the eigenvalues have a nonzero imaginary part. By
Jordan decomposition, it is easily seen that trajectories of $\Phi^i$ converge
in spiralling, so that the mappings
$\tau^i(x) = \inf\{t > 0: \: \Phi^i_t(x) \in \dR a\}$ and $h^i(x) = \Phi^i_{\tau^i(x)}$
are well defined and continuous. Let $H : \dR a \to \dR a$ be the map
$h^0 \circ h^1$ restricted to $\dR a.$ Since two different trajectories of the
same flow have empty intersection, the sequence $x_n = H^n(0)$ is
decreasing (for the ordering on $\dR a$ inherited from $\dR$). Being
bounded (recall that $M$ is compact and positively invariant), it converges
to $x^* \in \dR a$ such that $x^* = H(x^*)$.  Let now
$\gamma^0 = \{\Phi_t^1(x^*), \; 0 \leq t \leq \tau^1(x^*)\},
\gamma^1 = \{\Phi_t^0(h^1(x^*)),\; 0 \leq t \leq \tau^0(h^1(x^*))\}$
and $\gamma = \gamma^0 \cup \gamma^1.$ Reasoning as previously
shows that $\Gamma$ is the bounded component of
$\dR^2 \setminus \gamma$ and that there is a unique invariant and
absolutely continuous invariant probability.

\begin{figure}
\begin{center}
 \includegraphics[scale=0.5]{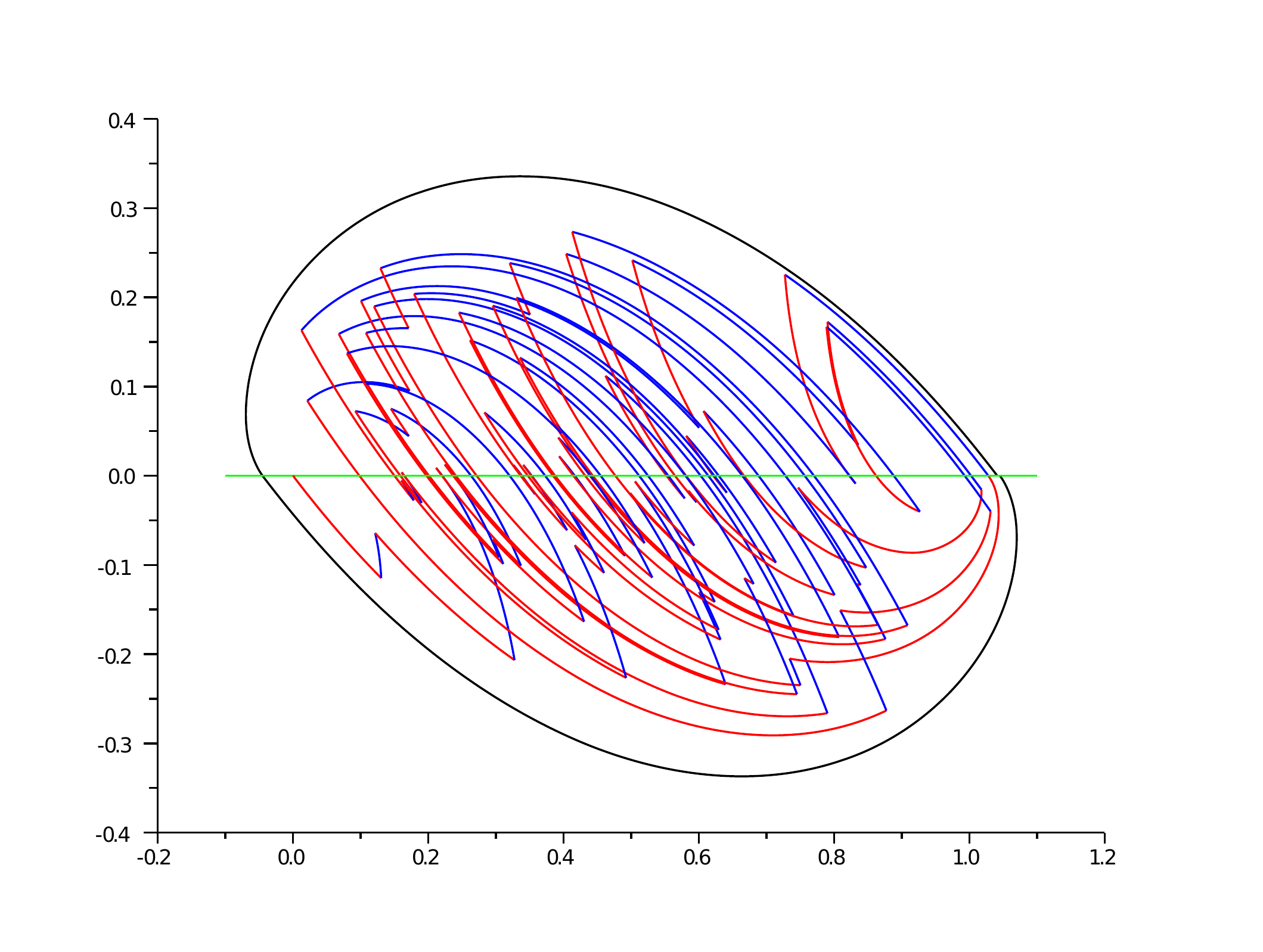}
 \caption{A sample path of $X_t$ (red and blue lines) and the boundary of the support of the 
 invariant probability (black line) for the third case in Section~\ref{ex:linear}, 
 where $A$ and $a$ are given by~\eqref{eq:ex-linear}.}
 \label{fi:drotation}
\end{center}
\end{figure}
We illustrate this situation in Figure~\ref{fi:drotation}, with
\begin{equation}\label{eq:ex-linear}
A =
\left(
\begin{array}{cc}
-1 & -1\\
1 &  -1
\end{array}
\right)
\quad\text{and}\quad
a=
\left(
\begin{array}{c}
1 \\
0
\end{array}
\right).
\end{equation}

\begin{rem}
Once again, if the jump rates are small,
then the density is unbounded at 0 but also on the set
\[
\BRA{\Phi^1_t(0),\ t\geq 0}\cup \BRA{\Phi^0_t(a),\ t\geq 0}.
\]
\end{rem}

\subsection{A simple criterion for the accessible set to have a non empty interior}

Here is a simple criteria in dimension 2 that ensures that $\Gamma$ has a 
non empty interior. 

\begin{prop}\label{prop:crit-2d}
Assume that $M\subset\dR^2$, $E=\BRA{0,1}$ and that $F^1$ 
has a globally attracting equilibrium $p$ such that the eigenvalues 
of $DF^1(p)$ have negative real parts and that $F^0(p) \neq 0.$ Then
$p$ lies in the interior of $\Gamma$.
\end{prop}
\begin{proof}
Proposition \ref{attractequil} ensures that $p$ belongs to $\Gamma$. As 
illustrated by Figure~\ref{fi:2d}, from the equilibrium $p$, one can follow $F^0$
and reach $x$, then follow $F^1$, and switch back to $F^0$ to reach any point 
in the shaded region.
\begin{figure}
\begin{center}
\begin{tikzpicture}[scale=2,decoration={markings}]
  \path (-1,0) to[out=5,in=175] node[coordinate,midway] (zero) {} (1,0);
  \node[coordinate] (one) at ($(zero)+(0.05,0.05)$) {};
  \begin{scope}
    \clip (-1,-0.4) to[out=5,in=175] (1,-0.4)
       -- (1,0.2)   to[out=175,in=5] (-1,0.2)
       -- cycle;
  \shade[left color=blue!50!red!50,]
  (1,0.5)
  \foreach \i in {2.5,2.55,...,3.2}{%
    --($(zero) !\i*\i! \i*180:(one)$)}
    -- ++(1,0) -- cycle;
  \end{scope}
  \foreach \y in {-1,-0.8,...,1}
  \draw[postaction={decorate,decoration={mark=at position .8 with {\arrow{>}}}}]
     (-1,\y) to[out=5,in=175] (1,\y);

  \draw[thick,postaction={decorate,decoration={%
    mark=between positions .2 and .8 step .2 with {\arrow{>}},
    reverse path}}]
  (zero)
  \foreach \i in {0,0.05,...,4}{%
  --($(zero) !\i*\i! \i*180:(one)$)};

  \fill (zero) circle (0.8pt) node[above right] {$p$};
  \fill (0.98,0)  circle (0.8pt) node[above left] {$x$};
\end{tikzpicture}
\caption{The flows near the attracting point $p$ in Proposition~\ref{prop:crit-2d}.}
  \label{fi:2d}
\end{center}
\end{figure}
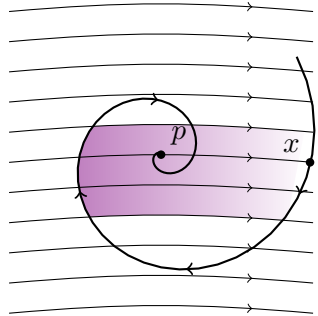
\end{proof}

\subsection{Knowing the flows is not enough}
\label{sec:ovidiu}

In this section we study a PDMP on $\dR^2\times \BRA{0,1}$ 
such that the strong bracket condition holds everywhere except on $\Gamma$ 
and which may have one or three ergodic invariant probability measures, 
depending on the jump rates of the discrete part of the process.

This model has been suggested by O.~Radulescu. The continuous
part of the process takes its values on $\dR^2$ whereas its discrete
part belongs to $\BRA{0,1}$. For sake of simplicity we will denote (in a
different way than in the beginning of the paper) by $(X_t,Y_t)\in\dR^2$
the continuous component. The discrete component ${(I_t)}_{t\geq 0}$
is a continuous time Markov chain on $E=\BRA{0,1}$ with jump
rates ${(\lambda_i)}_{i\in E}$. Let $\alpha>0$. The two vector fields $F^0$
and $F^1$ are given by
\[
F^0(x,y)=
\left(
\begin{array}{c}
 -x+\alpha\\
 -y+\alpha
\end{array}
\right)
\quad\text{and}\quad
F^1(x,y)=
\left(
\begin{array}{c}
 \displaystyle{-x+\frac{\alpha}{1+y^2}}\\
 \displaystyle{-y+\frac{\alpha}{1+ x^2}}
\end{array}
\right)
\]
with $(x,y)\in\dR^2$.
Notice that the quarter plane $(0,+\infty)^2$ is positively invariant by $\Phi^0$ 
and $\Phi^1$. See Figure~\ref{fi:ovidiu}. In the sequel we assume that $(X_0,Y_0)$ 
belongs to $(0,+\infty)^2$. 

\subsubsection{Properties of the two vector fields}

Obviously, the vector fields $F^0$ has a unique stable point $(\alpha,\alpha)$. 
The description of $F^1$ is more involved and depends on $\alpha$. 
\begin{lem}
Let us define
\begin{equation}\label{eq:def-a-b}
a=\frac{\alpha+\sqrt{\ABS{\alpha^2-4}}}{2}
\quad\text{and}\quad
b=\PAR{\frac{\sqrt{4/27+\alpha^2}+\alpha}{2}}^{1/3}
-\PAR{\frac{\sqrt{4/27+\alpha^2}-\alpha}{2}}^{1/3}.
\end{equation}
Notice that $b$ is positive and is the unique real solution of $b^3+b=\alpha$.
One has
\begin{itemize}
\item if $\alpha\leq 2$, then $F^1$ admits a unique critical point $(b,b)$
and it is stable,
\item if $\alpha>2$, then $F^1$ admits three critical points: $(b,b)$ is unstable
whereas $(a,a^{-1})$ and $(a^{-1},a)$ are stable.
\end{itemize}
\end{lem}

\begin{proof}
If $(x,y)$ is a critical point of $F^1$ then $(x,y)$ is solution of
\[
\begin{cases}
x(1+y^2)=\alpha\\
y(1+x^2)=\alpha.
\end{cases}
\]
As a consequence, $x$ is solution of
\[
0=x^5-\alpha x^4+2 x^3-2\alpha x^2+(1+\alpha^2)x-\alpha
=(x^2-\alpha x+1)(x^3+x-\alpha).
\]
The equation $x^3+x-\alpha=0$ admits a unique real solution $b$ given by \eqref{eq:def-a-b}. 
It belongs to $(0,\alpha)$. Obviously, if $\alpha\leq 2$, $(b,b)$ is the unique critical point
of $F^1$ whereas, if $\alpha>2$ then $a$ and $a^{-1}$ are the roots
of $x^2-\alpha x+1=0$ and $F^1$ admits the three critical points:
$(b,b)$, $(a,a^{-1})$ and $(a^{-1},a)$. Let us have a look to the stability
of $(b,b)$. 
The eigenvalues of $\mathrm{Jac}(F^1)(b,b)$ are given by
\[
\eta_1=-3+\frac{2b}{\alpha}=-1-2\frac{\alpha-b}{\alpha}
\quad\text{and}\quad \eta_2=
1-\frac{2b}{\alpha}=\frac{b^3-b}{\alpha}
\]
and are respectively associated to the eigenvectors $(1,1)$ and $(1,-1)$.
Since $b<\alpha$, $\eta_1$ is smaller than $-1$. Moreover, $\eta_2$ has the
same sign as $b-1$ \emph{i.e.} the same sign as $\alpha-2$. As a
conclusion, $(b,b)$ is stable (resp.\ unstable) if $\alpha<2$ (resp. $\alpha>2$).

Assume now that $\alpha>2$. Then
$\mathrm{Jac}(F^1)(a,a^{-1})$ has two negative eigenvalues $-1\pm 2\alpha^{-1}$.
Then, the critical points $(a,a^{-1})$ and $(a^{-1},a)$ are stable.
\end{proof}

In the sequel, we assume that $\alpha>2$.
One can easily check that the sets
\[
D=\BRA{(x,x)\, : \, x>0},\quad
L=\BRA{(x,y)\, :\, 0<y<x},\quad
U=\BRA{(x,y)\, :\, 0<x<y}
\]
are strongly positively invariant by $\Phi^0$ and $\Phi^1$. Moreover, 
thanks to Proposition~\ref{attractequil}, the accessible set of $(X,Y)$ is 
\[
\Gamma=\BRA{(x,x)\, :\, x\in[b,\alpha]}.
\]
In the sequel, we prove that $\Gamma$ may, or may not, be the set of all recurrent points,
depending on the jump rates $\lambda_0$ and $\lambda_1$.

\begin{prop}\label{prop:ex-rec-tra}
If $\lambda_1>\lambda_0(c\alpha-1)$, with $c=3\sqrt{3}/8$ then $(X,Y,I)$ admits a unique invariant measure 
and its support is $\Gamma\times E$. 

If $\lambda_1/\lambda_0$ is small enough, then $(X,Y,I)$ admits three ergodic measures 
and they are supported by 
\[
\Gamma \times E=\overline{\gamma^+(\alpha,\alpha)} \times E,
\quad
\overline{ \gamma^+(a,a^{-1})}\times E
\quad\text{and}\quad
\overline{\gamma^+(a^{-1},a)} \times E.
\]
\end{prop}
\begin{rem}
 This dichotomy is essentially due to the fact that the stable manifold $\BRA{(x,x)\ :\ x\in\dR}$ 
 of the unstable critical point $(b,b)$ of $F^1$ is strongly positively invariant by $\Phi^0$ 
 (see Figure~\ref{fi:ovidiu}). Moreover, $\overline{ \gamma^+(a,a^{-1})}\times E$ can be written 
 as the union of the segment $\SBRA{(a,a^{-1}),(\alpha,\alpha)}$ and $\Gamma$ and 
 the unstable manifold (included in $L$) of $(b,b)$ for $\Phi^1$. 
 \end{rem}

The following two sections are dedicated to the proof of Proposition~\ref{prop:ex-rec-tra}. 

\begin{figure}
\begin{center}
 \includegraphics[scale=0.5]{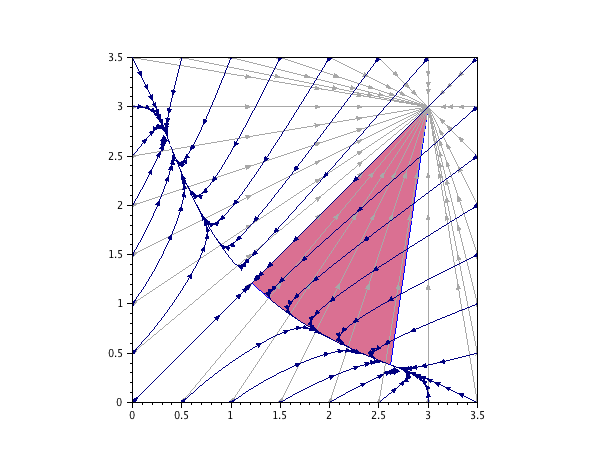}
 \caption{Vector fields $F^0$ (grey lines) and $F^1$ (blue lines) 
 of the example in Section \ref{sec:ovidiu} with $\alpha=3$. The shaded region is 
 $\overline{ \gamma^+(a,a^{-1})}\times E$.}
 \label{fi:ovidiu}
\end{center}
\end{figure}

\subsubsection{Transience}
The goal of this section is to prove the first part of Proposition~\ref{prop:ex-rec-tra}. 
\begin{lem}\label{le:majo}
Assume that $(X_0,Y_0)\in L$. Then, for any $t>0$,
\[
0< X_t-Y_t\leq (X_0-Y_0)\exp\PAR{-\int _0^t\! \alpha(I_s)\,ds},%
\]
with $\alpha(0)=1$ and $\alpha(1)=1-c\alpha<0$ with $c=(3/8)\sqrt{3}$.
\end{lem}
\begin{proof}
If $I_t=0$ then
\[
\frac{d}{dt}(X_t-Y_t)=-(X_t-Y_t).
\]
 On the other hand, if $I_t=1$ then
\begin{align*}
\frac{d}{dt}(X_t-Y_t)&=-(X_t-Y_t)+
\alpha \frac{X_t^2-Y_t^2}{(1+X_t^2)(1+Y_t^2)}\\
&=-\PAR{1-\alpha h(X_t,Y_t)}(X_t-Y_t)
\end{align*}
where the function $h$ is defined on $[0,\infty)^2$ by
\[
h(x,y)=\frac{x+y}{(1+x^2)(1+y^2)}.
\]
The unique critical point of $h$ on $[0,\infty)^2$ is $(1/\sqrt 3,1/\sqrt 3)$ and $h$
reaches its maximum at this point:
\[
c:=\sup_{x,y>0}h(x,y)=\frac{3\sqrt 3}{8}.
\]
As a consequence, for any $t\geq 0$,
\[
\frac{d}{dt}(X_t-Y_t)\leq -\alpha(I_t)(X_t-Y_t)%
\quad\text{where}\quad
\begin{cases}
\alpha(0)=1,\\
\alpha(1)=1-c \alpha.
\end{cases}
\]
Integrating this relation concludes the proof.
\end{proof}

\begin{cor}\label{cor:convpsD}
Assume that $(X_0,Y_0)\in L$. If $\lambda_1>\lambda_0(c\alpha-1)$ then
$(X_t,Y_t)$ converges exponentially fast to $D$ almost surely. More precisely,
\begin{equation}
  \label{eq:convpsD}
\limsup_{t\to \infty}\frac{1}{t}\log \PAR{X_t-Y_t}
\leq -\frac{\lambda_1-(c\alpha-1)\lambda_0}{\lambda_0+\lambda_1}
<0\quad a.s.
\end{equation}
In particular, the process $(X,Y,I)$ admits a unique invariant measure $\mu$. Its support is the set
\[
S=\BRA{(x,x,i)\,:\, x\in [b,\alpha],\ i\in\BRA{0,1}}.
\]
\end{cor}
\begin{proof}
The ergodic theorem for the Markov process ${(I_t)}_{t\geq 0}$
ensures that
\[
\frac{1}{t}\int_0^t \! \alpha(I_s)\,ds
\xrightarrow[t\to \infty]{a.s.} \int \! \alpha(i)d\nu(i)
\]
where the invariant measure $\nu$ of the process ${(I_t)}_{t\geq 0}$ is
the Bernoulli measure with parameter $\lambda_0/(\lambda_0+\lambda_1)$.
The upper bound \eqref{eq:convpsD} is a straightforward consequence
of Lemma~\ref{le:majo}. This ensures that the sets $L$ and $U$ are
transient.
\end{proof}

\subsubsection{Recurrence}
In this section, we aim at proving the second part of Proposition~\ref{prop:ex-rec-tra}. 
Let us define the following new variables:
\[
U_t=\frac{X_t+Y_t}{2}
\quad\text{and}\quad
V_t=\frac{X_t-Y_t}{2}.
\]
Of course $(U,V,I)$ is still a PDMP.
If
\[
\frac{d}{dt}
\PAR{
\begin{array}{c}
X_t \\
Y_t
\end{array}
}
=F^1(X_t,Y_t)
\quad\text{then}\quad
\frac{d}{dt}
\PAR{
\begin{array}{c}
U_t \\
V_t
\end{array}
}
=G^1(U_t,V_t),
\]
with
\[
G^1(u,v)=
\frac{1}{2}
\begin{pmatrix}
1 & 1 \\
1 & -1
\end{pmatrix}
F^1(u+v,u-v)
=
\begin{pmatrix}
-u+\dfrac{\alpha(1+u^2+v^2)}{(1+(u+v)^2)(1+(u-v)^2)}\\[3ex]
-v+\dfrac{2\alpha uv}{(1+(u+v)^2)(1+(u-v)^2)}
\end{pmatrix}
.
\]
Corollary~\ref{cor:convpsD} ensures that, if $\lambda_1/\lambda_0$ is
large enough, then $V_t$ goes to 0 exponentially fast.
Let us show that this is no longer true if $\lambda_1/\lambda_0$ is small enough.
Let $\varepsilon>0$. Assume that, with positive probability,
$V_t\in (0,\varepsilon)$ for any $t\geq 0$. Then, for any time
$t\geq 0$, $(U_t,V_t)\in [b,\alpha]\times [0,\varepsilon]$. Indeed, one can
show that $U_t\in[b,\alpha]$ for any $t\geq 0$ as soon as it is
true at $t=0$. 
The following lemma states that the vector fields $G^1$ can be compared to 
a vector fields $H^1$ (which is simpler to study).  
\begin{lem}\label{lem:compG}
 Assume that $(u,v)\in [b,\alpha]\times[0,\varepsilon]$. Then there exist
 $u_c\in(b,\alpha)$ and $K,\delta,\gamma,\tilde\gamma >0$ (that do not depend
 on $\varepsilon$) such that $b_\varepsilon=b+K\varepsilon^2$ and
 \[
G^1_1(u,v)\leq H^1_1(u,v)
\quad\text{with}\quad
H^1_1(u,v)=-\delta( u-b_\varepsilon).
\]
and
 \[
G^1_2(u,v)\geq H^1_2(u,v)
\quad\text{with}\quad
H^1_2(u,v)=
\PAR{%
  (\gamma+\tilde \gamma)
  \ind_\BRA{u\leq u_c}
 -\tilde\gamma} v.
\]
\end{lem}
\begin{proof}
Notice firstly that there exists $c>0$ such that 
\begin{equation}\label{eq:D}
\forall (u,v)\in [b,\alpha]\times[0,\varepsilon],\quad 
\ABS{(1+(u+v)^2)(1+(u-v)^2)-(1+u^2)^2}\leq c\varepsilon^2.
\end{equation}
Thus, using that $u^3+u-\alpha=(u-b)(u^2+bu+\alpha/b)$ and $u>b$ we get that
\begin{align*}
G^1_1(u,v)&\leq -u+\frac{\alpha}{1+u^2}+c\varepsilon^2 \\
&\leq -(u-b)\frac{u^2+bu+\alpha/b}{1+u^2}+c\varepsilon^2\\
&\leq -(u-b)\frac{2b^2+\alpha/b}{1+\alpha^2}+c\varepsilon^2.
\end{align*}
We get the desired upper bound for $G^1_1$ with
\[
\delta=\frac{2b^2+\alpha/b}{1+\alpha^2}, \quad K=c/\delta
\quad \text{and}\quad
b_\varepsilon=b+K\varepsilon^2.
\]
Similarly, Equation~\eqref{eq:D} ensures that
\[
G^1_2(u,v)\geq vg(u)
\quad\text{with}\quad
g(u)=\frac{2\alpha u}{(1+u^2)^2}-1-c\varepsilon^2.
\]
Obviously, if $\varepsilon$ is small enough, $g(b)>0$, $g(\alpha)<0$
and $g$ is decreasing on $[b,\alpha]$. Thus, if $\tilde u$ is the unique zero of $g$
on $(b,\alpha)$, then one can choose
\[
u_c=\frac{\tilde u+b}{2},\quad
\gamma=g(u_c)\quad\text{and}\quad
\tilde \gamma=g(\alpha).
\]
To get a simpler bound in the sequel we can even set
$\tilde \gamma=g(\alpha)\vee 1$.
\end{proof}
Finally, define $H^0_1(u,v)=G^0_1(u,v)=-(u-\alpha)$
and $H^0_2(u,v)=G^0_2(u,v)=-v$ and introduce the PDMP
${(\tilde U,\tilde V, \tilde I)}$ where $\tilde I=I$ is the switching process of
$(U,V,I)$ and $(\tilde U, \tilde V)$ is driven by $H^0$ and $H^1$
instead of $G^0$ and $G^1$. From Lemma~\ref{lem:compG}, we get
that
\[
\forall t\geq 0,\quad U_t \leq \tilde U_t
\quad\text{and}\quad
\tilde V_t\leq V_t
\]
assuming that $(\tilde U_0,\tilde V_0,\tilde I_0)=(U_0,V_0,I_0)$. The last
step is to study briefly the process $(\tilde U,\tilde V,\tilde I)$. From the definition 
of the vector fields that drive $(\tilde U,\tilde V,\tilde I)$, one has  
\[
\frac{d}{dt} \tilde V_t=
\begin{cases}
-\tilde V_t &\text{if }I_t=0,\\
((\gamma+\tilde\gamma)\ind_\BRA{\tilde U_t\leq u_c}-\tilde \gamma)V_t &\text{if }I_t=1. 
\end{cases}
\]
This ensures that 
\begin{equation}\label{eq:comp-V-tilde}
\frac{1}{t}\log \frac{\tilde V_t}{\tilde V_0}
\geq \frac{1}{t}\int_0^t
((\gamma+\tilde \gamma) \ind_\BRA{I_s=1,\tilde U_s\leq u_c}-\tilde \gamma)\,ds
\end{equation}
since $\tilde \gamma\geq 1$. If $\lambda_1/\lambda_0$ is small enough, then 
$(I_s,\tilde U_s)$ spends an arbitrary large amount of time near $(1,b_\varepsilon)$ 
(and it can be assumed that $b_\varepsilon< u_c$ if $\varepsilon$ is small enough). Thus,
the right-hand side of \eqref{eq:comp-V-tilde} converges almost surely
to a positive limit as soon as $\lambda_1/\lambda_0$ is small enough.
As a consequence, $V$ cannot be bounded by $\varepsilon$ forever. The Markov 
property ensures that $(X,Y)$ can reach any  neighborhood of $(a,a^{-1})$ with 
probability 1 and thus, $(a,a^{-1})$ belongs to the support of the invariant measure. 
This concludes the proof of the second part of Proposition~\ref{prop:ex-rec-tra}.

\section{Absolute continuity --- proofs of global criteria}
\label{absolute}
This section is devoted to the proof of 
Theorems~\ref{thm=globalRegJumps} and~\ref{thm=globalRegFixed}. 
The main idea of the proof has already been given before the statements; 
the main difficulty lies in providing estimates that are (locally) uniform in
the starting point $(x,i)$, the region of endpoints $\cV\times E$, and the discrete 
time~$m$ or the continuous time~$t$.

After seeing in Section~\ref{sec:inTermsOfVectorFields} what the 
submersion hypothesis means in terms of vector fields, 
we establish in Section~\ref{sec:paramLI} a parametrized version of the
local inversion lemma. This provides the uniformity in the continuous
part $x$ of the starting point, and enables us to prove 
in Section~\ref{sec:weakGlobal} a weaker version of Theorem~\ref{thm=globalRegJumps}. 
In Section~\ref{sec:gainUniform} we show how to prove the result
in its full strength. Finally, the fixed time result (Theorem~\ref{thm=globalRegFixed})
is proved in Section~\ref{sec:prfGlobalRegFixed}.

\subsection{Submersions, vector fields and pullbacks}
\label{sec:inTermsOfVectorFields}
Before going into the details of the proof, let us see how
one can interpret the submersion hypotheses of our regularity theorems. 

Recall that, for $x$ and $\vtr{i}$ fixed,  we are interested in the map
\begin{align*} 
  \psi : \dR^m &\to \dR^d \\
          \vtr{v}&\mapsto \PhiIV(x).
\end{align*}
To see if this is a submersion at $\vtr{u}$, we compute the partial derivatives
with respect to the $v_i$: these are elements of $T_{x_m}M$, and 
$\psi$ is a submersion if and only if these $m$ vectors span $T_{x_m}M$. 
This is the case if and only if their inverse image by $D\PhiIU$
span $T_{x_0}M$. An easy computation (see also Figure~\ref{fig:dessinGlobal})
shows that these vectors are given by:

  \begin{equation}
    \label{eq:defCtildeIU}
    \Cdiscret(\vtr{i},\vtr{u})  
    = \BRA{ F^{i_0}(x_0), \, \pullbk{1}F^{i_1} (x_0),\, \dots, \, \pullbk{m}F^{i_m} (x_0) },
  \end{equation}
  where $\pullbk{k}$ is the
  composite pullback:
\begin{equation}
  \pullbk{k} = \tpullbk{i_0}{u_1}\circ\cdots \circ \tpullbk{i_{k-1}}{u_k}.
  \label{eq:defPhiKStar}
\end{equation}
Note that $\pullbk{k}$ depends on  $\vtr{i}$ and $\vtr{u}$, 
but we hide this dependence for the sake of readability. 

\begin{figure}
\begin{tikzpicture}[%
  vector/.style={very thick,->},%
  dark blue/.style={blue!60!black},
  dark red/.style={red!60!black},
  dark green/.style={green!60!black},
  point/.style={circle,fill=black,inner sep=0pt,minimum size=3pt},
  scale = 0.7,
  baseline = (current bounding box.center)
  ]
  \path (0,0) node[coordinate] (x1s) {}
     +(1,1)   node[coordinate] (x1e) {};
  \path (5,2) node[coordinate] (x2s) {}%
     +(-1,-1) node[coordinate] (x2e){};
  \path (4,-2) node[coordinate] (x3s) {}%
     +(3,2) node[coordinate] (x3e){};
  \draw (x1s) .. controls (x1e) and +(-1,0)..  (x2s)%
    \foreach \p/\q in {0.2/1,0.4/2,0.6/3,0.8/4} {coordinate[pos=\p] (A\q)};
  \draw (x2s) .. controls (x2e) and +(-0.5,1)..  (x3s)%
    \foreach \p/\q in {0.2/1,0.4/2,0.6/3,0.8/4} {coordinate[pos=\p] (B\q)};
    \draw (x3s) .. controls (x3e) and +(-2,-1)..%
       +(5.5,2); 
  \foreach \q in {1,2,3,4} {
    \draw[vector,red!30] (A\q)-- + (0.2-0.3*\q,-0.70-0.1*\q);
    \draw[vector,green!30] (B\q)-- + (2+0.25*\q,0.7+0.3*\q);
    \draw[vector,green!30] (A\q)-- + (1+0.25*\q,2.5-0.4*\q);
    }
    \draw[vector,blue]  (x1s) --  (x1e) node [above right,dark blue] {$F_1$};
    \draw[vector,red]   (x2s) --  (x2e)   node [below left,dark red] {$F_2$};
    \draw[vector,green] (x3s) --  (x3e)  node [right,dark green] {$F_3$};
    \draw[vector,green] (x2s) -- +(2,0.7) node[above,dark green]
    {};
    \draw[vector,green] (x1s) -- +(1,2)   node[above,dark green]
    {$\pullbk{2} F_3$};
    \draw[vector,red]   (x1s) -- +(0.2,-0.7) node[below,dark red]
    {$\pullbk{1}  F_2$};

  \node[point, label=below:$x_2$] at (x3s) {};
  \node[point, label=left:$x_0$]  at (x1s) {};
  \node[point, label=above:$x_1$] at (x2s) {};
\end{tikzpicture}
\hfill
\begin{minipage}{0.4\linewidth}
  \small
  In this picture $(i_0,i_1,i_2) = (1,2,3)$.  The trajectory starts at $x_0$, 
  and follows
  $F^{i_0} = F^1$ for a time $u_1$. At the first jump, it starts
  following $F^{i_1} = F^2$; we pull this tangent vector, depicted in red,
  back to $x_0$. The next (green)
  tangent vector $F^{i_2} = F^3$ (at $x_2$) has to be pulled back by the two flows.
  If the three tangent vectors we obtain at $x_0$ span $T_{x_0}M$, 
  $\vtr{v}\mapsto \PhiIV(x_0)$ is a submersion.
\end{minipage}
\caption{The global condition}
\label{fig:dessinGlobal}
\end{figure}

\subsection{Parametrized local inversion}
\label{sec:paramLI}
Let us first prove a ``uniform'' local inversion lemma, for functions
of $\vtr{t}$ that depend on a parameter $x$.
\begin{rem}
  Even if $x$ lives in some $\dR^d$, we do not write it in boldface,
  for the sake of coherence with the rest of the paper.
\end{rem}
\begin{lem}
  \label{lem:local_inversion}
  Let $d$ and $m$ be two integers, and
 let $f$  be a $\cC^1$ map from $\dR^m \times \dR^d$ to $\dR^m$,
  \[ f: (\vtr{t},x) \mapsto f(\vtr{t},x) = f_x(\vtr{t}). \]
  For any fixed $x$, $f_x$ maps $\dR^m$ to itself; we denote
  its derivative at $\vtr{t}$ by $(Df_x)_{\vtr{t}}$.
  Suppose that, for some points $x_0$ and $\vtr{t}_0$, $(Df_{x_0})_{\vtr{t}_0}$
  is invertible. Then we can find  a neighborhood $J \subset \dR^d$ of $x_0$,
  an open set $I\subset \dR^m$ and, for all $x \in J$, an open set
  $W_x \subset \dR^m$, such that:
  \[
  \tilde{f}_x:\begin{cases}
    W_x &\to I, \\
    \vtr{t} &\mapsto f_x(\vtr{t})
  \end{cases}
    \]
    is a diffeomorphism. Moreover, for any integer $k\leq m$, and any
    neighborhood $W$ of $\vtr{t}_0$,
    we can choose $I$, $J$ and the $W_x$ so that:
    \begin{enumerate}[label=\roman*)]
    	\item \label{itemA}
     $I$ is a cartesian product $I_1 \times I_2$ where
     $I_1 \subset \dR^k$, $I_2 \subset \dR^{m-k}$ ;
   \item \label{itemB}
     $\forall x\in J, \quad W_x \subset W$.
    \end{enumerate}
\end{lem}
\begin{proof}
  We ``complete'' the map $f$ by defining:
  \[
  H: \begin{cases}
    \dR^m \times \dR^d &\to \dR^m \times \dR^d \\
    (\vtr{t},x) &\mapsto (f_x(\vtr{t}), x).
  \end{cases}
  \]
  The function $H$ is $\cC_1$, and its derivative can be written in block form:
  \[ DH_{(\vtr{t},x)} = \begin{pmatrix}
    (Df_x)_\vtr{t}  & \star \\
      0        &  I_k
    \end{pmatrix}.
    \]
    Since $(Df_{x_0})_{\vtr{t}_0}$ is invertible, $(DH)_{\vtr{t}_0,x_0}$ is invertible.
    We apply the local inversion theorem to $H$: there exist open sets
    $\cU_0$, $\cV_0$ such that $H$ maps $\cU_0$ to $\cV_0$ diffeomorphically.
    In order to satisfy the properties \ref{itemA} and \ref{itemB}, we restrict
    $H$ two times. First we define $\cU_1 = \cU_0 \cap (W\times \dR^d)$,
    and $\cV_1 = H(\cU_1)$.
    Since $\cV_1$ is open it contains a product set $\cV = I_1 \times I_2\times J$,
    and we let $\cU = H^{-1}(\cV)$. For any $(y,x)\in I \times J$, define
    $g_x(y)$ the first component of $H^{-1}(y,x)$: composing by $H$, we see
    that $f_x(g_x(y)) = y$.

    The set $W_x = \{ \vtr{t}\in \dR^m ; (\vtr{t},x) \in \cU \}$
    is open, and included in $W$. Since $f_x$ maps $W_x$ to $I$, $g_x$
    is its inverse and both are continuous, so $\tilde{f}_x$
    is a diffeomorphism.
\end{proof}

\begin{lem}
  \label{lem:rankToRegularity}
  Let $T$ be a continuous random variable in $\dR^m$, with density $h_T$.
  Let $d\leq m$, and let $\phi$ be a $\cC^1$ map from $\dR^m \times \dR^d$
  to $\dR^d$:
  \[
  \phi: (\vtr{t},x) \mapsto \phi_x(\vtr{t}).
  \]
  Suppose that, for some $x_0, \vtr{t}_0$,
  $(D\phi_{x_0})_{\vtr{t}_0}: \dR^m \to \dR^d$ has full rank $d$.
  Suppose additionally that $h_T$ is bounded below by $c_0>0$ on a
  neighborhood of $\vtr{t}_0$.

  Then there exist a constant $c>0$, a neighborhood $J$ of $x_0$
   and a neighborhood $I_1$ of
   $\phi_{x_0}(\vtr{t}_0)$ such that:
   \begin{equation}
     \label{eq=lemmeCalculDiff}
     \forall x \in J, \quad \prb{ \phi(T,x) \in \cdot } \geq c \leb{d}(\cdot \cap I_1).
   \end{equation}
  In other words, $\phi(T,x)$ has an absolutely continuous part with respect to the
  Lebesgue measure.
\end{lem}

\begin{proof}
  We know that $(D\phi_{x_0})_{\vtr{t}_0}$ has rank $d$. Without
  loss of generality, we suppose that the first $d$ columns are independent.
  In other words, writing $\vtr{t} = (\vtr{u},\vtr{v}) \in \dR^d \times \dR^{m-d}$,
  we suppose that the derivative of
  $\psi_{x,\vtr{v}}: \vtr{u} \mapsto \phi_{x_0}(\vtr{u},\vtr{v})$
  is invertible in $\vtr{u}_0$ for $\vtr{v} = \vtr{v}_0$.

  Once more, we ``complete'' $\phi$ and define:
  \[
  f_x:
  \begin{cases} \dR^d \times \dR^{m-d} &\to \dR^d \times \dR^{m-d} \\
    (\vtr{u},\vtr{v})                  &\mapsto (\phi_x(\vtr{u},\vtr{v}), \vtr{v}).
  \end{cases}
  \]
   By Lemma \ref{lem:local_inversion} applied with $k=d$, we can find
  $I_1\subset \dR^d$, $I_2\subset \dR^{m-d}$, $J\subset \dR^d$ and
  $(W_x)_{x\in J} \subset \dR^m$ such that $f_x$ maps diffeomorphically
  $W_x$ to $I_1\times I_2$. Call $\tilde{f_x}$ this diffeomorphism.
  By property \ref{itemB} of the lemma, we can ensure that  $W_x$ is
  included in a given neighborhood of $\vtr{t}_0$. Since
  $Df_x = \begin{pmatrix} D\psi_{x,\vtr{v}}  &\star \\ 0 & I \end{pmatrix}$,
  we can choose this neighborhood so that:
  \begin{equation}
    \label{eq=inversible}
    \forall x\in J, \forall \vtr{t} \in W_x, \quad
    h_T(\vtr{t}) \ABS{\det((Df_x)_{\vtr{t}})}^{-1} \geq c' >0.
\end{equation}
  for some strictly positive constant $c'$.

Write the random variable $T$ as a couple $(U,V)$, and let $A$ be a Borel
set included in~$I_1$.
  \begin{align*}
    \prb{\phi(T,x) \in A}
    &\geq \prb{\phi(T,x) \in A, V \in I_2} \\
    &=    \prb{f_x(U,V) \in A \times I_2} \\
    &\geq  \prb{ (U,V) \in \tilde{f_x}^{-1} (A \times I_2)} \\
    &= \int_{\tilde{f}_x^{-1}(A \times I_2)} h_T(\vtr{u},\vtr{v})d\vtr{u}d\vtr{v} \\
    &= \int_{\tilde{f}_x^{-1}(A \times I_2)} h_T(\vtr{u},\vtr{v})
  \ABS{\det( (D\tilde{f}_x)_{\vtr{u},\vtr{v}})}^{-1}\cdot
  \ABS{\det( (D\tilde{f}_x))_{\vtr{u},\vtr{v}}} d\vtr{u}d\vtr{v}.
  \end{align*}
  Since $\tilde{f}_x^{-1} (A\times I_2) \subset W_x$, we may use the
  bound \eqref{eq=inversible}.
  Then we can change variables by defining
  $(\vtr{s},\vtr{v}) = \tilde{f}_x(\vtr{u},\vtr{v})$. We obtain:
  \begin{align*}
    \prb{\phi(T,x) \in A}
    &\geq c' \int_{\tilde{f}_x^{-1}(A \times I_2)}
    \ABS{\det( (D\tilde{f}_x))_{\vtr{u},\vtr{v}}} d\vtr{u}d\vtr{v} \\
    &=    c' \int_{A\times I_2}  d\vtr{s} d \vtr{v} \\
    &\geq c' \leb{d}(A) \leb{m-d}(I_2).
  \end{align*}
  Therefore \eqref{eq=lemmeCalculDiff} holds with $c = c'\leb{m-d}(I_2)$.
\end{proof}

\subsection{A slightly weaker global condition}
\label{sec:weakGlobal}
\begin{prop}[Regularity at jump times --- weak form]
  \label{prop:weakGlobal}
  Let $x_0$ be a point in $M$, and $(\vtr{i},\vtr{u})$ an adapted sequence in $\dT_m$,
  such that $\min_{0\leq i \leq m}u_i>0$. 

  If $\vtr{v}\mapsto \PhiIV(x_0)$ is a submersion at $\vtr{u}$,
   then there exists $\mathcal{U}_0$ a neighborhood of $x_0$,
   $\mathcal{V}_0$ a neighborhood of $\PhiIU(x_0)$  and a constant $c>0$ such that:
   \begin{equation}
     \forall x\in \cU_0, \quad
     \prb[x,i_0]{\tZ_m \in \cdot \times \{i_m\}}
     \geq c \leb{d}\PAR{ \cdot \cap \cV_0},
     \label{eq:weakRegAtJumpTimes}
   \end{equation}
   where $i_0$ and $i_m$ are the first and last elements of $\vtr{i}$. 
 \end{prop}

 \begin{rem}
   This result is a weaker form of Theorem~\ref{thm=globalRegJumps}:
   \begin{itemize}
     \item the hypothesis is stronger --- the sequence $(\vtr{i}, \vtr{u})$ 
       must be adapted with strictly positive terms;
     \item the conclusion is weaker~--- it lacks uniformity for the discrete
       component, for the starting point and the final point. 
   \end{itemize}
 \end{rem}

  \begin{proof}
Recall $(U_i)_{i \geq 1}$ is the  sequence of interarrival times
of a homogeneous Poisson process.
Let $\tribuDuPoisson$ be the sigma field generated by $(U_i)_{i \geq 1}$.
Set $\vtr{U} = (U_1,\ldots,U_m)$ and $\tilde{\vtr{Y}} = (\tY_0,\ldots, \tY_m).$
By continuity, there exists a neighborhood $\cU_0$ of $x_0$, and numbers $\delta_1,\delta_2 > 0$ such that $p(x,\vtr{i},\vtr{v}) \geq \delta_2$ for all $x \in \cU_0$ and $\vtr{v} \in \dR^m$ such that
$\NRM{\vtr{v}-\vtr{u}} =  \max_{1\leq i \leq m} |v_i - u_i| \leq \delta_1$. 
Therefore
\begin{align*}
  \prb[x,i_0]{\tX_m \in \cdot, \, \tY_m = i_m}
  & \geq \prb{\bphi^{\vtr{i}}_{\vtr{U}}(x) \in \cdot, \;
              \tilde{\vtr{Y}} = \vtr{i}, \; 
	      \NRM{\vtr{U}-\vtr{u}} \leq \delta_1 }          \\
  & =  \esp{\prb{\bphi^{\vtr{i}}_{\vtr{U}}(x) \in \cdot, \; 
                   \tilde{\vtr{Y}} = \vtr{i}, \; 
		   \NRM{\vtr{U}-\vtr{u}} \leq \delta_1 
		 \middle| \tribuDuPoisson}} \\
  & \geq \delta_2 \prb{\bphi^{\vtr{i}}_{\vtr{U}}(x) \in \cdot, \; 
                       \NRM{\vtr{U}-\vtr{u}} \leq \delta_1}  \\
  & = \delta_2 \delta_3 \prb{\bphi^{\vtr{i}}_{\vtr{U}}(x) \in \cdot 
                             \middle| \NRM{\vtr{U}-\vtr{u}} \leq \delta_1 }  \\
  &   = \delta_2 \delta_3 \prb{\bphi^{\vtr{i}}_{\vtr{T}}(x) \in \cdot} 
\end{align*}
where $\delta_3 = \prb{\NRM{\vtr{U}-u} \leq \delta_1} > 0$ and 
 $\vtr{T} = (T_1,\ldots,T_m)$ is a vector of independent random variables such that for each $i$,
the distribution of $T_i$ is given by
\[
  \prb{T_i < t} = \prb{U_i < t \:\big|\: \ABS{U_i - u_i}\leq \delta_1}.
\]
On $[u_i - \delta_1,u_i + \delta_1]$ this is
\(
 \frac{e^{\lambda \delta_1} - e^{-\lambda(t-u_i)}}{e^{\lambda \delta_1} - e^{-\lambda \delta_1}}\)
so $T_i$ has the density
\begin{equation}
  \label{eq=densiteConditionnelle}
  f_{T_i}(t) = \ind_{[u_i - \delta_1,u_i+\delta_1]}(t)
               \frac{\lambda e^{-\lambda (t-u_i)}}
	       { e^{\lambda \delta_1} - e^{-\lambda \delta_1}}
	     \end{equation}
which is continuous at the point $u_i$. 

  Lemma~\ref{lem:rankToRegularity} then applies, 
  yielding \eqref{eq=regAtJumpTimes}, with $\cU_0$ and $\cV_0$ given
  by $J$ and $I_1$ of Lemma~\ref{lem:rankToRegularity}.
  \end{proof}

\subsection{Gaining uniformity}
\label{sec:gainUniform}
\subsubsection{Uniformity at the beginning}
The uniformity on the discrete component follow from two main
ideas:
\begin{itemize}
	\item use the irreducibility and aperiodicity to 
	   move the discrete component, 
	\item use the finite speed given by compactness to show that 
	  this can be done without moving too much. 
\end{itemize}

The vector fields $F^i$ are continuous and the space is compact, so
the speed of the process is bounded by a constant $C_{sp}$. 

\begin{defi}
  [Shrinking] For any open set $\cU$ and any $t>0$, define  $\cU_t$ 
  the shrunk set:
  \begin{equation}
    \label{eq:defShrink}
    \cU_t = \BRA{ x\in A, d(x, \cU^c) >  C_{sp} t }
  \end{equation}
  This set is open, and non empty for $0<t<t(\cU)$. 
  If $x\in\cU_t$, then $\prb[x,i]{X_t \in \cU} = 1$. 
\end{defi}
\begin{lem}[Uniformity at the beginning]
  \label{lem:uniformiteInitiale}
  Let $\cU$ be a non empty open set. There exist $0<\epsilon_1<\epsilon_2$,
  an integer $m_b$, an open set $\cU'\subset \cU$ and a constant~$c$ such that:
  \begin{align*}
  \forall x\in \cU', \forall i,j, &&
  \prb[x,i] {\forall t \in [\epsilon_1,\epsilon_2], Z_t \in \cU\times\{ j\}} &\geq c,   \\
  \forall x\in \cU', \forall i,j &&
  \prb[x,i]{\tZ_{m_b}\in \cU \times\{j\}} &\geq c.
\end{align*}
\end{lem}
\begin{proof}
  Let $\epsilon_2 < t(\cU)$, $\cU' = \cU_{\epsilon_2}$ and
$\epsilon_1 = \epsilon_2/2$.  There is a positive probability
that between $t=0$ and $t = \epsilon_1$, the index jumps  from $i$ to $j$,
and does not jump again before time $t=\epsilon_2$; the fact that
$X_t\in\cU$ is guaranteed by the definition of $\cU'$. 
The second result is similar; if all jump rates are positive, we
can even choose  $m_b = 1$.
\end{proof}

\subsubsection{Uniformity at the end}

\begin{lem}[Gain of discrete uniformity]
  \label{lem:uniformiteFinale}
  If $\cV$ is an open set and $i\in E$, 
  then there exist $c'$, $\cV'$, $t_e$ and $m_e$ such that,
  if $\mu \geq c (\lambda_\cV \times \delta_i)$, 
  \begin{align*}
    \mu P_{t_e}       &\geq c'(\lambda_{\cV'\times E}), \\
    \mu \tP^{m_e} &\geq c'(\lambda_{\cV'\times E}).
  \end{align*}
\end{lem}

The proof will use the following result:
\begin{lem}[Propagation of absolute continuity]
  \label{lem:propagation}
  There is a constant $C_{div}$ that only depends on the set $M$
  and the vector fields $F^i$ such that, for all $\cV$, $i$, 
  \[ 
    \PAR{ \lambda_\cV\times \delta_i }K_t \geq e^{-C_{div} t} \PAR{\lambda_{\cV_t}\times \delta_i},
  \]
  where $\cV_t$ is the shrunk set defined in~\eqref{eq:defShrink} and $K_tf(x,i) = \Phi_t^i f(x,i)$. 
\end{lem}
\begin{proof}
 Since $\Phi_t^i$ is a diffeomorphism 
  from $\dR^d$ to itself, for any positive map $f$ on $\dR^d\times E$ we get by change of variables:
  \[  \int f(\Phi^i_t(x),i) d\lambda_{\cV}(x)  = \int f(x,i) \ABS{D\Phi^i_t}^{-1} d\lambda_{\Phi^i_t(\cV)}(x). \]
  If we let $h(t) = \ABS{D\Phi^i_t}$, one of the classical interpretation of the divergence operator
  (see e.g.~\cite{Lee13}, Proposition~16.33) yields $h'(t) = h(t)\Div F^i (\Phi^i_t(x))$. By compactness, 
  \[ \exists C_{div}, \forall x\in M, \forall i, \quad \ABS{\Div F^i(x)} \leq C_{div}.\]
  Therefore $h(t)^{-1} \geq \exp( - C_{div} t)$. 

  Since by definition of the shrunk set, 
  $\Phi_t^i(\cV) \supset \cV_t$, 
  \[  \int f(\Phi^i_t(x),i) d\lambda_{\cV}(x)  \geq \exp(-C_{div}t) \int f(x,i) d\lambda_{\cV_t}(x), \]
  and Lemma~\ref{lem:propagation} follows. 
\end{proof}

\begin{proof}
  [Proof of Lemma~\ref{lem:uniformiteFinale}]
  Fix a point $x\in\cV$. Since the matrix $Q(x) = Q(x,i,j)$ is irreducible 
  and aperiodic, there exists an integer $m$ such that for all $i$, $j$, 
  there exists  a sequence $\vtr{i}(i,j) = (i_0 = i,i_1(i,j), \ldots , i_{m - 1}(i,j), i_m(i,j) = j)$ that
  satisfies $\prod_{l=1}^m Q(x, i_{l-1}(i,j), i_l(i,j)) > 0$. Without loss
  of generality (since we can always replace $\cV$ by a smaller set) we suppose that
\begin{equation}
\label{eq:boundOnQs}
\forall x\in \cV, \forall i,j, \forall l, \quad Q(x, i_{l-1}(i,j), i_l(i,j)) \geq c_Q > 0. 
\end{equation}
Fix $i$ and $j$. From~\eqref{eq:def2Pt} we can rewrite $P_t$ as:
\begin{equation*}
P_t = \sum_{n \geq 0} \lambda^n e^{-\lambda t}
\int_{\{\vtr{u} \in \dR^n : \sum_{i= 1}^n u_i < t\}}
      \PAR{K_{u_1} Q K_{u_2} Q \cdots  K_{u_n} Q K_{t - \sum u_i}} du_1 \ldots du_n.
\end{equation*}
Therefore:
  \begin{align*}
    \PAR{\lambda_{\cV}\otimes\delta_i} P_t
    &\geq \lambda^m e^{-\lambda t}\int_{\vtr{u}\in\dR^m: \sum u_i < t} 
    (\lambda_{\cV}\otimes \delta_i) K_{u_1} Q \cdots K_{u_m}QK_{t - \sum u_i} du_1\cdots du_n.
  \end{align*}
  By Lemma~\ref{lem:propagation} and the lower bound~\eqref{eq:boundOnQs}, 
  \begin{align*}
    (\lambda_{\cV}\otimes \delta_i) K_{u_1} Q
    &\geq e^{-C_{div} u_1} (\lambda_{\cV_{u_1}}\otimes \delta_i)Q \\
    &\geq c_Q e^{-C_{div} u_1} (\lambda_{\cV_{u_1}}\otimes \delta_{i_1}). 
  \end{align*}
  Repeating these two lower bounds $m$ times yields:
  \begin{align*}
    \PAR{\lambda_{\cV}\otimes\delta_i} P_t
    &\geq \lambda^m e^{-\lambda t} e^{-C_{div} t} c_Q^m \PAR{\int_{\vtr{u}\in\dR^m: \sum u_i < t} 
  du_1\cdots du_n} \lambda_{\cV_t}\otimes \delta_j \\
  &\geq \frac{(\lambda c_Q t)^m}{m!}  e^{-(\lambda+C_{div} ) t} \lambda_{\cV_t}\otimes \delta_j.
  \end{align*}

  Since the measures $\lambda_{\cV_t}\otimes \delta_j$ are mutually singular for different indices $j$, 
  this implies that
  \begin{align*}
    \PAR{\lambda_{\cV}\otimes\delta_i} P_t
    &\geq c(\lambda,t,c_Q,m) \lambda_{\cV_t \times E}.
  \end{align*}
  For $t$ small enough, $\cV_t=\cV'$ is non empty, and the first part of the lemma follows.  

  The  statement for $\tP^m$ is proved similarly, starting from the bound
  \[ \tP^m \geq 
\int_{\{\vtr{u} \in \dR^m : \sum_{i= 1}^m u_i < t\}}
      \PAR{K_{u_1} Q K_{u_2} Q \cdots  K_{u_n} Q } du_1 \ldots du_m,
    \]
    written for a $t$ small enough so that $\cV_t$ is non empty.  
\end{proof}

\subsubsection{Proof of Theorem~\ref{thm=globalRegJumps}}
  The hypothesis gives the existence of $(\vtr{i},\vtr{u})$ such that
  $\vtr{v}\mapsto \PhiIV(x_0)$ is a submersion at $\vtr{u}$, or in 
  other words that the family $\Cdiscret(\vtr{i},\vtr{u})$ defined 
  by \eqref{eq:defCtildeIU} has full rank. 
  If $(\vtr{i},\vtr{u})$ is not adapted to $x_0$,
 by the irreducibility hypothesis, there exists an $m$ and a sequence $(\vtr{i}',\vtr{u}') \in \TT_m$ 
 such that $(\vtr{i}',\vtr{u}')$ is adapted and describes the same trajectory (just add instantaneous transitions where 
 it is needed). 
 The new family $\Cdiscret(\vtr{i}',\vtr{u}')$ contains all vectors from $\Cdiscret(\vtr{i},\vtr{u})$,
 so
 \[
 \rk(\Cdiscret(\vtr{i}',\vtr{u}')) \geq \rk(\Cdiscret(\vtr{i},\vtr{u})).
 \]
 Now, for any $m$, the mapping:
   $(\vtr{i},\vtr{u})  \mapsto \Cdiscret(\vtr{i},\vtr{u})$ from 
   $\TT_m$ to $(\dR^d)^m$ is continuous. Since the rank is a lower semicontinuous function, the mapping
 \begin{align*} 
   \TT_m &\to \dN   \\
   (\vtr{i},\vtr{u})  &\mapsto \rk\PAR{\Cdiscret(\vtr{i},\vtr{u})}
 \end{align*}
 is lower semi-continuous.
 Since being adapted is an open condition, there exists a sequence $(\vtr{i}'',\vtr{u}'') \in \TT_m$ such that
   every component of $\vtr{u}''$ is strictly positive, and
   $\rk\PAR{\Cdiscret(\vtr{i}'',\vtr{u}'')} \geq \rk\PAR{\Cdiscret(\vtr{i}',\vtr{u}')}$. 

 In other words, if the submersion hypothesis of Theorem~\ref{thm=globalRegJumps}
 holds, then the stronger hypothesis of Proposition~\ref{prop:weakGlobal} 
 holds for a (possibly longer) adapted sequence with non-zero terms.

 By Proposition~\ref{prop:weakGlobal}, there exists $\cU$, $\cV$ and~$c$ such that
 \[
   \forall x\in \cU, \quad \prb[x,i_0]{ \tZ_m \in \cdot \times\{i_m\}} \geq c\lambda_\cV(\cdot).
 \]
 Using Lemma~\ref{lem:uniformiteInitiale} to gain uniformity at the beginning, 
 we get the existence of $m'=m_b + m$, $\cU'$ and~$c'$ such that:
 \[
   \forall x\in \cU',\forall i\in E, \quad \prb[x,i]{ \tZ_{m'} \in \cdot \times\{i_m\}} \geq c'\lambda_\cV(\cdot),
 \]
 or in other words:
 \[
   \forall x\in \cU', \forall i\in E,\quad (\delta_{x,i})\tP^{m'} \geq c' \lambda_\cV\otimes\delta_{i_m}.
 \]
 Finally we apply Lemma~\ref{lem:uniformiteFinale} to $i=i_m$ and  the 
 measure $\mu=\delta_{x,i} \tP^{m'}$ to get uniformity at the end:
 for $m'' = m' + m_e$, 
 \[ \forall x\in \cU', \forall i\in E,\quad (\delta_{x,i}) \tP^{m''} \geq c'' \lambda_{\cV' \times E},\]
 which is exactly the conclusion of Theorem~\ref{thm=globalRegJumps}.

\subsection{Absolute continuity at fixed time}%
\label{sec:prfGlobalRegFixed}

The hypothesis of Theorem~\ref{thm=globalRegFixed} is that
$\psi:\vtr{v}\mapsto \Phi_{t-\sum v_i}^i \circ \PhiIV$
has full rank at $\vtr{u}$. 
Reasoning as in Section~\ref{sec:inTermsOfVectorFields}, we can compute 
the derivatives with respect to 
the $v_i$, and write the rank condition at the initial point $x_0$:
the submersion hypothesis holds if and only if the family
  \begin{equation}
  \label{eq:defCIU}
  \begin{split}
    \Ccontinu(\vtr{i},\vtr{u}) =
    \left\{ \PAR{F^{i_0}                 -  \pullbk{m}F^{i_{m}}}(x_0), 
          \PAR{\pullbk{1}F^{i_1}       -  \pullbk{m}F^{i_{m}}}(x_0),  \right. \\
	  \left.\cdots,
	  \PAR{\pullbk{m-1}F^{i_{m-1}} -  \pullbk{m}F^{i_{m}}}(x_0)
	\right\}
  \end{split}
   \end{equation}
   has full rank.

  Let us now turn to the proof of \eqref{eq=regAtFixedTimes},

  By the same continuity arguments as above, we suppose without loss of generality
  that the sequence $(\vtr{i},\vtr{u})$ is adapted to $x_0$ and that all 
  elements of $\vtr{u}$ are positive. 
  Moreover,  there exist $\cU_0$, $\delta_1$ and 
  $\delta_2$ such that, if $x\in\cU_0$ and $\vtr{v}\in\dR^m$ satisfies
  $\NRM{\vtr{v} - \vtr{u}} \leq \delta_1$, then $\sum v_i < t_0$ and
  $p(x, \vtr{v}, \vtr{i} )\geq \delta_2$. 
  
  Define two events
  \begin{align*}
    A &= \text{``the process jumps exactly $m$ times before time $t_0$''}
         \cap  \BRA{ \NRM{\vtr{U} - \vtr{u}} \leq \delta_1 }, \\
    B &= \BRA{ \tilde{\vtr{Y}} = \vtr{i} }
  \end{align*}
  The event $A$ is $\tribuDuPoisson$-measurable.  By definition of $\delta_1$, $\delta_2$, 
  \begin{align*}
    \ind_{A} \prb[x,i_0]{B| \tribuDuPoisson} &\geq \delta_2 \ind_{A} \\
    &\geq \delta_2 \ind_{\BRA{U_{k+1} >t_0}} \ind_{\BRA{\NRM{\vtr{U} - \vtr{u}} \leq \delta_1}} \\
    &\geq \delta_2 e^{-\lambda t_0} \ind_{\BRA{\NRM{\vtr{U} - \vtr{u}} \leq \delta_1}}. 
\end{align*}
On the event $B$, $Z_{t_0} = (\psi(\vtr{U}),i_m)$, so:
  \begin{align*}
    \prb[x,i_0]{Z_{t_0} \in \cdot\times\BRA{i_m}}
    &\geq \prb[x,i_0]{A\cap B \cap \PAR{Z_{t_0} \in \cdot\times\BRA{i_m} }} \\
    &=    \prb[x,i_0]{A \cap B \cap  (\psi(\vtr{U}) \in \cdot)} \\
    &=    \esp{ \prb[x,i_0]{B|\tribuDuPoisson} \ind_A \ind_{\psi( \vtr{U}) \in \cdot }} \\
    &\geq \delta_2e^{-\lambda t_0} \prb{ \BRA{\NRM{\vtr{U} - \vtr{u}}\leq \delta_1} \cap \psi(\vtr{U}) \in \cdot} \\
    &\geq \delta_2\delta_3 e^{-\lambda t_0} \prb{ \psi(\vtr{U}) \in \cdot \big|  \NRM{\vtr{U}-\vtr{u}} \leq \delta_1}.
  \end{align*}
  where $\delta_3 = \prb{\NRM{\vtr{U}-\vtr{u}} \leq \delta_1}$. The reasoning leading to Equation~\eqref{eq=densiteConditionnelle}
  still applies. Thanks to
Lemma~\ref{lem:rankToRegularity}, this implies \eqref{eq=regAtFixedTimes},
but only with $i = i_0$, $j=i_m$ and $\epsilon = 0$.

To prove the general form of \eqref{eq=regAtFixedTimes} with the
additional freedom in the choice of  $i$, $j$  and $t$, we  
first use Lemma~\ref{lem:uniformiteInitiale} to find a
neighborhood $\cU'_0$ of $x_0$, and three constants
$0<\epsilon_1<\epsilon_2$ and $c>0$  such that:
  \[
  \forall x\in \cU'_0,
  \quad
  \prb[x,i]{ \forall t\in[\epsilon_1,\epsilon_2],  Z_t \in \cU_0 \times\{ i_0\}}
  \geq c.
  \]
  Let $t_0' = t_0 - \epsilon_1$ and $\epsilon = \epsilon_2 - \epsilon_1$, so
  that $[t'_0,t'_0+ \epsilon] = [t_0 + \epsilon_1, t_0 + \epsilon_2]$.
  Then, for any $x\in \cU'_0$, and any $t\in[t'_0, t'_0 + \epsilon]$,
  \begin{align*}
    \prb[x,i]{X_t \in \cdot }
    &\geq \esp[x,i]{
       \ind_{\{Z_{t-t_0} \in \cU_0 \times\{ i_1\}\}}
       \prb[Z_{t-t_0}]{ X_{t_0} \in \cdot}
     } \\
     &\geq c'\leb{m}( \cdot \cap \cV_0).
\end{align*}
An application of Lemma~\ref{lem:uniformiteFinale} proves that we can also gain uniformity
at the end; this concludes the proof of Theorem~\ref{thm=globalRegFixed}.

\section{Constructive proofs for the local criteria}
\label{alamain}

\subsection{Regularity at jump times}
  To prove the local criteria, we show that they imply the global ones for appropriate 
  (and small) times $u_1, \dots u_m$. We introduce some additional notation 
  for some families of vector fields. 
  \begin{defi}
    The round letters $\famF$, $\famG$, $\famH$ will denote families of vector fields on $M$.
    For a family $\famF$, $\famF_x$ is the corresponding family of tangent vectors at $x$. 

    If $\vtr{i} =(i_0, \ldots,i_m)$ is a sequence of indices
    and $\vtr{u} = \vtr{u}(t) = (u_1(t),\ldots u_m(t))$ is a sequence of ``time'' functions, 
    we denote by $\FIU  = \FIU(t)$ the family of vector fields:
\begin{equation}
  \left\{ F^{i_0}, \pullbk{1}F^{i_1}, \dots \pullbk{m} F^{i_m}\right\}.
  \label{eq:definition_FIU}
\end{equation}
This family depends on $t$ via  the $\pullbk{k}$ (see \eqref{eq:defPhiKStar}).  
  \end{defi}

  We begin by a simple case where there are just two vector fields, 
  $F^1$ and $F^2$, and we want regularity at a jump time, starting
  from (say) $(x,1)$. To simplify matters further, suppose that the 
  dimension $d$ is two. 

  In the simplest case, $F^1(x)$ and $F^2(x)$ span the tangent plane $\dR^2$. 
  Then, for $t$ small enough, these vectors
  ``stay independent'' along the flow of $X_2$: $F^2(x)$ and $(\tpullbk{2}{t} F^1)(x)$ 
  are independent. So the global condition holds for $t$ small enough. 

  \bigskip
    To understand where Lie brackets enter the picture, let us first recall
    that they appear as a Lie derivative that describes how $X$ changes when
    pulled back by the flow of $F^i$ for a small time: at any given point $x$, 
    \[\lim_{t\to 0} \tpullbk{i}{t}X(x) - X(x) -  t[F^i,X](x)  = 0.\]
    Staying at a formal level for the time being, let us write this as:
    \begin{equation}
      \label{eq:lie_derivative_informal}
      \tpullbk{i}{t}X =  X -  t[F^i,X](x)  +o(t).
    \end{equation}
  Suppose now that $F^1$ and $F^2$ are collinear at $x$, but that $F^1(x)$ 
  and $[F^1,F^2](x)$ span $\dR^2$. We have just seen that: 
  \[
  \tpullbk{2}{t}(F^1) = F^1 + t [F^2,F^1] + o(t).
  \]
Let $\vtr{u}(t) = (t,t)$ and $\vtr{i} = (1,2,1)$, and look at $\FIU(t)$. 
  If the ``$o(1)$'' terms behave as expected, 
  \begin{align*}
    \FIU(t) 
    &= ( F^1,  \; \tpullbk{1}{t}(F^2),       \;  \tpullbk{1}{t} \tpullbk{2}{t} (F^1)) \\
    &= \PAR{F^1, \; F^2 + t[F^1,F^2] + o(t), \; F^1 + t[F^2,F^1]+o(t) 
    + t[F^1,F^1] + t^2[F^1,[F^2,F^1]] + o(t^2)} \\
    &= \PAR{F^1, \, F^2 + t[F^1,F^2] + o(t),\,  F^1 + t[F^2,F^1] + o(t)}. 
\end{align*}
By hypothesis,  $\rk\PAR{F^1(x),[F^2,F^1](x)} = 2$. 
  Therefore, for $t$ small enough, the lower-semicontinuity of the rank ensures: 
  \begin{align*}
    \rk\PAR{\FIU(t)}
    &= \rk\PAR{F^1,\; F^2 + t[F^1,F^2] + o(t), \;  t[F^2,F^1] + o(t)}\\
    &= \rk\PAR{F_1,\; F_2 + t[F^1,F^2] + o(t), \; [F^2,F^1] + o(1)}\\
    &\geq \rk( F^1,\; F^2, \; [F^2,F^1]) \\
    &= 2,
  \end{align*}
  and the global condition holds. 

  \bigskip
    To treat the general case, let us first define the $o(1)$ notation.   For
    any smooth vector field $X = \sum_i X^i(x)\frac{\partial}{\partial x_i}$,
    let
    \[ \NRM{X}_k = \max_{\alpha, \ABS{\alpha} \leq k} \max_{x\in M} \max_i \ABS{\partial^\alpha X^i (x)}.\]
    If $X(t)$ is a family of vector fields depending
    on the parameter~$t$, we write $X(t) = o(1)$ if
    \begin{equation}
      \label{eq:defPetitO}
      \forall k\geq 0, \quad \lim_{t\to 0}\NRM{X(t)}_k = 0.
    \end{equation}
  
  The previous case shows two main ingredients in the proof: 
  \begin{itemize}
  	\item to introduce the brackets, we have to alternate between flows, 
	\item the method works because we can express various vectors and Lie brackets
	  as (approximate) linear combinations of vectors in $\FIU(t)$, for good choices
	  of $\vtr{u}$ and $\vtr{i}$. 
  \end{itemize}
  Let us abstract the second ingredient in a definition. 
  \begin{defi}
    Let $\famG = \{G_1, \ldots G_n\}$ be a fixed family of vector fields and
    $\famH(t)$ a family depending on $t$. If there exist continuous functions 
    $\lambda_{ij}:(0,\infty) \to \dR$, and  vector fields $R_i(t)$, such that:
    \[ 
    \forall i\leq n,\quad G_i = \sum \lambda_{ij}(t) H_{j}(t) + R_i(t),
    \]
    and $R_i(t) =  o(1)$, we say that $\famH(t)$ 
    (asymptotically) generates $\famG$. 
  \end{defi}
  \begin{rem}
    We allow the $\lambda_{ij}(t)$ to blow up when $t\to 0$. Because of this,
    even if the $H_j(0)$ are defined, $G_i$ does not necessary lie in their
    linear span. For example, if $H_1(t) = (1,0)$ and $H_2(t) = (0,t)$,
    $\famH(t)$ asymptotically generates any family of constant vector fields in
    $\dR^2$, even if $H_2(0) = (0,0)$ degenerates.  
  \end{rem}
  \begin{lem}
    \label{lem:asymptoticRank}
    If $\famH(t)$ generates $\famG$, then for any point $x$, there is a $t_x$ such that:
    \[ \forall t < t_x, \quad \rk(\famH(t)_x) \geq \rk(\famG_x).\]
  \end{lem}
  This result will be proved below. The last ingredient in the proof is to 
  introduce different time scales in the alternation between vector fields. 
  Let us write $u(t)\ll v(t)$ if $u(t) = o (v(t))$ when $t$ goes to zero. 
  \begin{lem}[Towers of Hanoï]
    \label{lem:hanoi}
    Suppose that, for some $\vtr{i} = (i_0, \ldots i_m)$ 
    and $\vtr{u}(t) = (u_1(t), \ldots u_m(t))$,
    $\FIU(t)$ generates $\famG$. Suppose that for all $j$, $u_j(t)$ takes positive 
    values and $\lim_{t\to 0} u_j(t) = 0$. 
    Choose two functions $u(t)$ and $v(t)$  such that $v(t) \ll u(t)$,
    $u(t)  \ll 1 $ and  $u_j(t) \ll  u(t)$, for all $j$. 
  Define $\tilde{\vtr{u}}(t)$ and $\tilde{\vtr{i}}$ by concatenation (see Figure~\ref{fig:hanoi}):
     \begin{equation}
       \label{eq:defiTilde}
     \begin{aligned}
       \tilde{\vtr{u}}(t) &= (u_1(t), u_2(t), \ldots, u_m(t), \; v(t),u(t), \; u_1(t), \ldots , u_m(t)); \\
       \tilde{\vtr{i}}    &= (i_0, i_1, \ldots i_m, i, i_0, i_1, \ldots i_m). 
     \end{aligned}
     \end{equation}
    
    Then $\FIUTilde(t)$ generates $\famG \cup \{F^i\} \cup \{ [F^i,G], G\in \famG\}$.
  \end{lem}
  The name comes from the fact that, in analogy with the towers of Hanoï, in order to gain brackets by $F^i$ (the lower disk),
  we have to move according to $\vtr{i},\vtr{u}$ (move all the upper disks), 
  then move the lower disk, then move the upper disks once again. 

  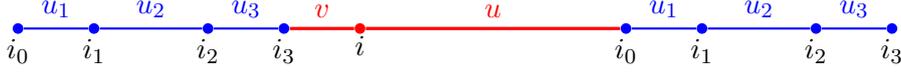
\begin{figure}
    \centering
  \begin{tikzpicture}
    \path[every node/.style={circle,fill=blue,minimum size=4pt,inner sep = 0pt}]
          (0,0) node[label=below:$i_0$] (i0) {}
          (1,0) node[label=below:$i_1$] (i1) {}
          (2.5,0) node[label=below:$i_2$] (i2) {}
          (3.5,0) node[label=below:$i_3$] (i3) {}
          (4.5,0) node[label=below:$i$,fill=red] (i) {}
          (8,0) node[label=below:$i_0$] (i0') {}
          (9,0) node[label=below:$i_1$] (i1') {}
	  (10.5,0) node[label=below:$i_2$] (i2') {}
          (11.5,0) node[label=below:$i_3$] (i3') {};
	  \draw[blue,thick] (i0)  -- node[midway,above]{$u_1$}
	                    (i1)  -- node[midway,above]{$u_2$}
			    (i2)  -- node[midway,above]{$u_3$} (i3)
                            (i0') -- node[midway,above]{$u_1$}
			    (i1') -- node[midway,above]{$u_2$}
			    (i2') -- node[midway,above]{$u_3$} (i3');
	   \draw[red,very thick] (i3) -- node[midway,above]{$v$}
	                          (i)  -- node[midway,above]{$u$} (i0');
  \end{tikzpicture}
    \caption{The Hanoï construction}
    \small
    The sequence $(\tilde{\vtr{i}} ; \tilde{\vtr{u}})$ is obtained by 
    inserting the sequence $(i ; (u,v))$, depicted in red, in-between two copies of the original sequence $(\vtr{i} ; \vtr{u})$. 
    The ``middle'' time $u$ is much larger than all the other times. 
    \label{fig:hanoi}
  \end{figure}

  Once these lemmas are known, Theorem \ref{th:hormander} follows quite easily. 
\begin{proof}
  [Proof of Theorem~\ref{th:hormander}] 
  Let us reason by induction. Starting from the empty family, we can
  construct thanks to Lemma~\ref{lem:hanoi} a $\vtr{u}(t)$ and an $\vtr{i}$
  such that $\FIU(t)$ generates any number of $F^i$ and iterated brackets
  $[F^{i_1}, [F^{i_2},\ldots F^{i_k}]]$.
  If $\dR^d$ is generated by these brackets, Lemma~\ref{lem:asymptoticRank}
  shows that $\FIU(t)$ has full rank for some $t_x$.  Therefore, the global
  condition is satisfied for the choice of times $(u_1(t_x),u_2(t_x), \ldots
  u_m(t_x))$. 
\end{proof}

Let us turn to the proofs of the lemmas. 
\begin{proof}[Proof of Lemma \ref{lem:asymptoticRank}]
  Define a family $\tilde{\famH}(t)$ by $\tilde{H}_i = \sum \lambda_{ij}(t) H_j(t) = G_i - R_i(t)$. 
  At every point $x$, every vector in $\tilde{\famH}_x$ is a combination of vectors in $\famH_x$. Therefore
  \begin{align*}
    \rk(\famH(t)) &\geq \rk(\tilde{\famH}(t)) \\
    &\geq \rk(\famG),
  \end{align*}
  where the second line follows from the lower semi-continuity of the rank, 
  since for all $x$, $\NRM{R_i(t)(x)}$ converges to $0$. 
\end{proof}

To prove Lemma~\ref{lem:hanoi} we need to write down properly the fundamental relation~%
\eqref{eq:lie_derivative_informal} and see how ``small'' vector fields are affected by pullbacks
and Lie brackets.
\begin{lem}
      \label{lem:contPB}
      For any $k$, there exists a constant $C_k$ that only depends
      on the fields $F^i$, such that the following holds. 
      For any vector field $X$, any $i$ and any $t\in [0,1]$, 
      \begin{align}
	\label{eq:contPB1}
	\NRM{ \tpullbk{i}{t} X}_k &\leq C_k \NRM{X}_k, \\
	\label{eq:contPB2}
	\NRM{ \tpullbk{i}{t} X - X }_k &\leq C_k t \NRM{X}_{k+1}, \\
	\label{eq:contPB3}
	\NRM{\tpullbk{i}{t}X - X - t[F^i,X]}_k  &\leq C_k t^2 \NRM{X}_{k+2}.
      \end{align}
      In particular if $X$ is smooth, the formal equation \eqref{eq:lie_derivative_informal}
      is rigorous. 

      Moreover, if $X(t) = o(1)$, then $\tpullbk{i}{t}X(t) = o(1)$ and $[F^i,X(t)] = o(1)$. 
    \end{lem}
\begin{proof}
  Let $F$ be one of the $F^i$, and $(t,x) \mapsto \Phi_t(x)$ be its flow. 
  Recall that the pullback $\Phi_t^\star$ acts on vector fields by:
  \[ (\Phi_t^\star X)(x) = (D\Phi_t)_x^{-1}  X(\Phi(t,x)). \]
  Since the flow is smooth, this can be written in coordinates as
\begin{equation}
\label{eq:pb_in_coordinates}
  (\Phi_t^\star X)(x) = \sum_{i,j} a^i_j(x,t)X^j(\Phi(t,x)) \frac{\partial}{\partial x_i},
\end{equation}
  where the $a^i_j$ are smooth functions of $x$ and $t$ that only depend on the 
  vector field $F$.  For any $i$ and any multiindex $\alpha$, we may apply
  $\partial^\alpha$ to the $i$\textsuperscript{th} coordinate; the resulting expression only 
  involves the derivatives of $X^j$ up to order $k$. This implies the control~\eqref{eq:contPB1}.
  
  To prove \eqref{eq:contPB2}, first apply Taylor's formula
  in the $t$ variable to each coordinate in \eqref{eq:pb_in_coordinates}, at order $1$:
  \begin{align*}
    &a^i_j(x,t)X^j(\Phi(t,x)) - X^j(x)\\
    &\qquad = \int_0^t \partial_s (a^i_j(x,s)X^j(\Phi(s,x))) ds \\
    &\qquad = \int_0^t (\partial_s a^i_j)(x,s)X^j(\Phi(s,x)) 
    + a^i_j(x,s) \left( \nabla X^j(\Phi(s,x)) \cdot  F(\Phi(s,x))\right) ds.
  \end{align*}
  This expression involves the $X^j$ and their first-order derivatives. Once more we
  may apply $\partial^\alpha$ to both sides to deduce \eqref{eq:contPB2}. The proof of%
  ~\eqref{eq:contPB3} is similar; the fact that the first order term is given by 
  the Lie bracket is standard.
  This clearly implies
  \eqref{eq:lie_derivative_informal} if $X$ is smooth. 

  Finally suppose $X(t)$ satisfies $X(t) = o(1)$. For any $k$, 
  $\NRM{ \tpullbk{i}{t} X(t)}_k \leq C_k \NRM{X(t)}_k$ by~\eqref{eq:contPB1}. 
  When $t$ goes to zero, $\NRM{X(t)}_k$ converges to zero, and so does
  $\NRM{\tpullbk{i}{t} X(t)}_k$: in other words,  $\tpullbk{i}{t}X(t) = o(1)$. 
  The fact that the same happens for Lie brackets follows from their expression
   in coordinates. 
\end{proof}

\begin{proof}[Proof of Lemma~\ref{lem:hanoi}]
  Recalling the composite pullback notation $\pullbk{k}$ from
  \eqref{eq:defPhiKStar}, 
  let us define, for $0\leq j \leq m$,  $\tilde{F}_j = \pullbk{j} F^{i_j}$
  so that $\FIU(t) = (\tilde{F}_0, \tilde{F}_1, \ldots \tilde{F}_m)$. 
  All these quantities, as well as $u$, depend on $t$, 
  but we drop this dependence in the notation. 
  The ``Hanoï'' construction yields:
  \[
  \begin{matrix}
    \FIUTilde = 
    \big( & F^{i_0}, & \pullbk{1} F^{i_1},& \ldots ,&\pullbk{m} F^{m}, \\[1.3ex]
    & & \pullbk{m}\tpullbk{i_m}{v} F^i, && \\[1.3ex]
    & \pullbk{m}\tpullbk{i_m}{v}\tpullbk{i}{u} F^{i_0}, 
    & \pullbk{m}\tpullbk{i_m}{v}\tpullbk{i}{u}\pullbk{1} F^{i_1},
	  &\ldots, 
	  &\pullbk{m}\tpullbk{i_m}{v}\tpullbk{i}{u} \pullbk{m} F^{i_m} \big) \\[3ex]
   \phantom{\FIUTilde}
  = \big( & \tilde{F}_0, &\tilde{F}_1,&  \ldots,& \tilde{F}_m,\\[1.3ex]
  &&\pullbk{m} \tpullbk{i_m}{v} F^i, &&\\[1.3ex]
  & \pullbk{m}\tpullbk{i_m}{v}\tpullbk{i}{u} \tilde{F}_0,
  & \pullbk{m}\tpullbk{i_m}{v}\tpullbk{i}{u} \tilde{F}_1,
	  &\ldots,
  & \pullbk{m}\tpullbk{i_m}{v}\tpullbk{i}{u} \tilde{F}_m\big).
  \end{matrix}
  \]
  Consider first the middle term. 
  Since $v = v(t)$ goes to zero, 
  \(\tpullbk{i_m}{v} F^i = F^i + o(1)\) by~\eqref{eq:contPB2} from Lemma~\ref{lem:contPB}. 
  Since $u_j(t)$ goes to zero, the same argument and the fact that pullbacks of $o(1)$ terms
  stay $o(1)$, once more by Lemma~\ref{lem:contPB}, show that
  $\tpullbk{i_{j-1}}{u_j}(F^i + o(1)) = F^i + o(1)$. 
  Therefore:
  \begin{align*}
    \pullbk{m}\tpullbk{i_m}{v} F^i &= F^i + o(1)
  \end{align*}
  so that $\FIUTilde(t)$ generates $F^i$.

  Now consider $\famG = (G_1, \ldots G_L)$. Since $\FIU(t)$ generates $\famG$,
  and $\FIUTilde(t)$ contains $\FIU(t)$, $\FIUTilde(t)$ generates all $G_l$. 
  Let us now prove that $\FIUTilde(t)$ generates $[F^i,G_l]$. 
  The first step is to write down a relation for $G_l$:
  \[
  G_l = \sum_j \lambda_{lj}(t) \tilde{F}_j(t) + R_l(t).
\]
where $R_l(t) = o(1)$. 
  Since this is an equality of vector fields, we can pull it back 
  by~$\pullbk{m}\tpullbk{i_m}{v}\tpullbk{i}{u}$:
  \[
    \pullbk{m}\tpullbk{i_m}{v}\tpullbk{i}{u} G_l
    = \sum_j \lambda_{lj}(t)\pullbk{m}\tpullbk{i_m}{v}\tpullbk{i}{u} \tilde{F}_j + \pullbk{m}\tpullbk{i_m}{v}\tpullbk{i}{u}R_l(t).
   \]
   The difference between the last two equalities yields
\begin{equation}
  \label{eq:almostThere}
    \pullbk{m}\tpullbk{i_m}{v}\tpullbk{i}{u} G_l - G_l
    = \sum_j \lambda_{lj}(t)(\pullbk{m}\tpullbk{i_m}{v}\tpullbk{i}{u} \tilde{F}_j  - \tilde{F}_j(t))
    + \pullbk{m}\tpullbk{i_m}{v}\tpullbk{i}{u}R_l(t) - R_l(t). 
  \end{equation}
Since the $u_j$ and $v$ are negligible with respect to $u$, repeated
applications of Lemma~\ref{lem:contPB} show that the left hand side can be written as
\[
    \pullbk{m}\tpullbk{i_m}{v}\tpullbk{i}{u} G_l - G_l
    = u[F^i,G_l] + o(u).
\]
Similarly the last term on the right hand side satisfies
\[
   \pullbk{m}\tpullbk{i_m}{v}\tpullbk{i}{u}R_l(t) - R_l(t)
   = u[F^i,R_l(t)] + o(u).
\]
Plugging this back in \eqref{eq:almostThere} and dividing by $u$ we get

\[
  [F_i,G_l] = \sum_j \frac{\lambda_{lj}(t)}{u} 
  (\pullbk{m}\tpullbk{i_m}{v}\tpullbk{i}{u} \tilde{F}_j  - \tilde{F}_j(t))
  + o(1) + [F_i,R_l(t)].
\]
Since $R_l(t) = o(1)$, the last statement of Lemma~\ref{lem:contPB} implies
that $[F^i,R_l(t)] = o(1)$.  Therefore $[F^i,G_l]$ can be written as the sum of 
a linear combination of vector fields in $\FIUTilde$, and a remainder. This 
shows that $\famG \cup \{F^i\} \cup \{ [F^i,G], G \in \famG \}$ is generated 
by $\FIUTilde$, and concludes the proof of Lemma~\ref{lem:hanoi}.
\end{proof}

\subsection{Regularity at a fixed time}
Once more, we show that if the local criterion holds, then the global one holds for a good choice
of indices and times. 
The global criterion is expressed in terms of the family described by \eqref{eq:defCIU}. 
It will be easier to work with a slightly different family,  namely:
\begin{equation}
  \GIU = \left(
  F^{i_0} - \pullbk{1}F^{i_1}, \pullbk{1} F^{i_1} - \pullbk{2} F^{i_2}, 
  \ldots,  \pullbk{m-1}F^{i_{m-1}} - \pullbk{m}  F^{i_{m}} \right). 
\end{equation}
It is easy to see that $\GIU$ and the original family span the same space at each point (the $k$\textsuperscript{th} vector 
in the original family is the sum of the last $m-k+1$ vectors of $\GIU$). The analogue of 
Lemma~\ref{lem:hanoi} is the following:
  \begin{lem}[More Towers of Hanoï]
    Suppose that, for some $\vtr{i}$ and some time functions $\vtr{u}(t)$, $\GIU(t)$ asymptotically generates $\famG$.
    Choose $u(t)$, $v(t)$ as in Lemma~\ref{lem:hanoi} and define $\tilde{\vtr{i}}$, $\tilde{\vtr{u}}(t)$ by concatenation
    as in \eqref{eq:defiTilde}.

     Then $\GIUTilde(t)$ asymptotically generates $\famG \cup \{F^i - F^{i_m}, F^i - F^{i_0}\} \cup \{ [F^i,G], G\in \famG\}$.  
  \end{lem}
  \begin{proof}
    For $1\leq j \leq m$, call $\tilde{F}_j$ the $j$\textsuperscript{th} vector field
    in $\GIU$: $\tilde{F}_j = \pullbk{j-1} F^{i_{j-1}} - \pullbk{j} F^{i_{j}}$. 
    The family $\GIUTilde$ is then
  \[
  \begin{matrix}
    \GIUTilde  
  = \big( & \tilde{F}_1, &\tilde{F}_2,&  \ldots,& \tilde{F}_m,\\[1.3ex]
  & \pullbk{m}     ( F^{i_{m}} - \tpullbk{i_{m}}{v} F^i),
  &  \pullbk{m} \tpullbk{i_m}{v} \PAR{ F^i - \tpullbk{i}{u} F^{i_0}},&&\\[1.3ex]
  &  \pullbk{m} \tpullbk{i_m}{v}\tpullbk{i}{u} \tilde{F}_1, 
  &  \pullbk{m} \tpullbk{i_m}{v}\tpullbk{i}{u} \tilde{F}_2, 
	  &\ldots,
  &  \pullbk{m} \tpullbk{i_m}{v}\tpullbk{i}{u} \tilde{F}_m, 
	  \big).
  \end{matrix}
  \]
  The two vector fields in the middle give, at zero-th order as~$t$ goes to
  zero, $F^{i_m} - F^i$ and $F^i - F^{i_0}$.  For any element $G_l\in\famG$, 
  we may repeat the exact same argument as in Lemma~\ref{lem:hanoi}
  to generate the vectors $G_l$ and $[F^i,G_l]$ from the family $\GIUTilde(t)$. 
  \end{proof}
  With this lemma in hand, we know we can generate $F^1 - F^2$ (starting from $\vtr{i} =(i_0) = (1)$, an empty $ \vtr{u}= ()$, 
  and choosing $i=2$). 
  In all successive ``enrichments'' of $\vtr{i},\vtr{u}$ by the Hanoï procedure, 
  the first and last components of $\vtr{i}$ will always be $1$, therefore
 given enough enrichments, we generate all the $F^i - F^1$.  Consequently  we also get all differences: $F^i - F^j = F^i - F^1 + F^1 - F^j$. 
  Finally, by taking the bracket by $F^i$, we generate $[F^i,F^j]$, and then all subsequent higher order brackets. 
  This concludes the proof of Theorem~\ref{th:hormander}.

\paragraph*{Acknowledgements}
FM and PAZ thank MB for his kind hospitality
and his coffee breaks. We acknowledge financial support from the Swiss National Foundation Grant 
FN 200021-138242/1
and the French ANR projects EVOL, ProbaGeo and ANR-12-JS01-0006 - PIECE.
This work was mainly done while PAZ (resp. FM) held a position at the University of Burgundy 
(resp. Rennes).

\addcontentsline{toc}{section}{\refname}%

{\footnotesize
 \bibliographystyle{amsplain}
\bibliography{flots}
}

\bigskip


{\footnotesize %
   \noindent Michel~\textsc{Bena\"im},
e-mail: \texttt{michel.benaim(AT)unine.ch}

 \medskip

\noindent\textsc{Institut de Math\'ematiques, Universit\'e de Neuch\^atel,
11 rue \'Emile Argand, 2000 Neuch\^atel, Suisse.}

\bigskip

   \noindent St\'ephane \textsc{Le Borgne},
e-mail: \texttt{stephane.leborgne(AT)univ-rennes1.fr}

 \medskip

 \noindent\textsc{IRMAR UMR 6625, CNRS-Université de Rennes 1,
   Campus de Beaulieu, 35042 Rennes \textsc{Cedex}, France.}

 \bigskip

 \noindent Florent \textsc{Malrieu},
 e-mail: \texttt{florent.malrieu(AT)univ-tours.fr}

 \medskip

 \noindent\textsc{LMPT UMR 7350, CNRS-Université de Tours,
   UFR Sciences et Techniques,
Parc de Grandmont,
37200 Tours, France.}

\bigskip

   \noindent Pierre-Andr\'e~\textsc{Zitt},
e-mail: \texttt{pierre-andre.zitt(AT)univ-mlv.fr}

 \medskip

 \noindent\textsc{LAMA UMR 8050, CNRS-Université-Paris-Est-Marné-La-Vallée, 
 5, boulevard Descartes,
 Cité Descartes, Champs-sur-Marne,
 77454 Marne-la-Vallée Cedex 2, France.}

}

\end{document}